\documentclass{amsart}
\usepackage{euscript}
\usepackage{pstricks}
\usepackage{amscd}
\usepackage{amssymb}
\allowdisplaybreaks
\numberwithin{equation}{section}
\numberwithin{figure}{section}

\def\Atilde{\widetilde{A}}

\def\height{\mathop{\hbox{\rm ht}}\nolimits}

\let\iso\cong

\def\Re{\mathop{\rm Re}\nolimits}
\def\Aut{\mathop{\hbox{\rm Aut}}\nolimits}

\def\calC{\EuScript{C}}

\def\halfX{{\textstyle\frac{1}{2}}X}
\def\halfTpX{\frac{1}{2}T_p X}

\def\O{{\rm O}}
\def\PU{{\rm PU}}

\def\M{\EuScript{M}}

\def\reverse{\mathop{\hbox{\rm reverse}}\nolimits}

\makeatletter
\def\dodgelinewidth{.4pt}
\def\dodgeradius{2}

\def\spaceabovedodgeline{1pt}
\def\spacebelowdodgeline{1.2pt} 
\def\dodge#1{\mathop{\vbox{\m@th\ialign{##\crcr\noalign{\kern\spaceabovedodgeline}\dodgefill\crcr\noalign{\kern\spacebelowdodgeline\nointerlineskip}$\hfil\displaystyle{#1}\hfil$\crcr}}}\limits}
\def\dodgefill{$\m@th
\setbox0=\vbox to\dodgelinewidth{}
\leaders\vrule height\ht0 depth0pt\hfill\dodgebump\leaders\vrule height\ht0 depth0pt\hfill$}
\def\dodgebump{{\psset{unit=1pt}\begin{pspicture}(-\dodgeradius,0)(\dodgeradius,0)\psarc[linewidth=\dodgelinewidth](0,0){\dodgeradius}{0}{180}\end{pspicture}}}
\makeatother

\makeatletter
\def\Looplinewidth{\dodgelinewidth}
\def\Loopradius{\dodgeradius}
\def\spaceabovemeridianline{\spaceabovedodgeline}
\def\spacebelowmeridianline{3pt}
\def\Loop#1{\mathop{\vbox{\m@th\ialign{##\crcr\noalign{\kern\spaceabovemeridianline}\Loopfill\crcr\noalign{\kern\spacebelowmeridianline\nointerlineskip}$\hfil\displaystyle{#1}\hfil$\crcr}}}\limits}
\def\Loopfill{$\m@th
\setbox0=\vbox to\Looplinewidth{}
\leaders\vrule height\ht0 depth0pt\hfill\Looploop$}
\def\Looploop{{\psset{unit=1pt}\begin{pspicture}(-\Loopradius,0)(\Loopradius,0)\psarc[linewidth=\Looplinewidth](0,0){\Loopradius}{0}{360}\end{pspicture}}}
\makeatother

\makeatletter
\def\geodesiclinewidth{\dodgelinewidth}

\def\spaceabovegeodesicline{\spaceabovedodgeline}
\def\spacebelowgeodesicline{1pt}
\def\geodesic#1{\mathop{\vbox{\m@th\ialign{##\crcr\noalign{\kern\spaceabovegeodesicline}\geodesicfill\crcr\noalign{\kern\spacebelowgeodesicline\nointerlineskip}$\hfil\displaystyle{#1}\hfil$\crcr}}}\limits}
\def\geodesicfill{$\m@th
\setbox0=\vbox to\geodesiclinewidth{}
\leaders\vrule height\ht0 depth0pt\hfill\geodesicarrow$}
\def\geodesicarrow{\vrule height\geodesiclinewidth width1pt\kern-1pt
\smash[tb]{\lower2.29pt\hbox{$\rightharpoonup$}}}
\makeatother

\let\geodesic\overline

\def\complexgeodesic#1{\geodesic{#1}{}^{\C}}

\def\cong{\equiv}
\def\orb{{\rm\scriptstyle orb}}
\def\piorb{\pi_1^\orb}
\def\Im{\mathop{\rm Im}\nolimits}
\def\Isom{\mathop{\rm Isom}\nolimits}

\def\PGL{\mathop{\rm PGL}\nolimits}
\def\Leech{\Lambda}
\def\w{\omega}
\def\d{\delta}

\def\wbar{\bar{\w}}

\def\zetabar{\bar{\zeta}}
\def\E{\EuScript{E}}
\def\Z{\mathbb{Z}}
\def\Q{\mathbb{Q}}
\def\F{\mathbb{F}}
\def\R{\mathbb{R}}
\def\C{\mathbb{C}}

\def\ch{\C H}

\def\chn{\ch^n}

\def\e{\varepsilon}
\def\H{\EuScript{H}}
\def\tensor{\otimes}
\def\sset{\subseteq}
\def\G{\Gamma}
\def\PG{P\G}
\def\thetabar{\bar{\theta}}
\def\mbar{\bar{m}}

\def\ip#1#2{\langle#1\,{|}\,#2\rangle}

\def\biggend#1{\bigl\langle#1\bigr\rangle}

\def\Bigip#1#2{\Bigl\langle#1\Bigm|#2\Bigr\rangle}
\def\cell{\bigl(\begin{smallmatrix}0&\thetabar\\\theta&0\end{smallmatrix}\bigr)}

\newtheorem{theorem}{Theorem}[section]
\newtheorem{lemma}[theorem]{Lemma}

\newtheorem{conjecture}[theorem]{Conjecture}

\theoremstyle{remark}
\newtheorem{remark}[theorem]{Remark}
\begin{document}
\title{Geometric generators for braid-like groups}
\author{Daniel Allcock}
\thanks{First author supported by NSF grant DMS-1101566}
\address{Department of Mathematics\\University of Texas, Austin}
\email{allcock@math.utexas.edu}
\urladdr{http://www.math.utexas.edu/\textasciitilde allcock}
\author{Tathagata Basak}
\address{Department of Mathematics, Iowa State University, Ames IA, 50011.}
\email{tathastu@gmail.com}
\urladdr{https://orion.math.iastate.edu/tathagat/}

\date{January 12, 2015}

\begin{abstract}
We study the problem of finding generators for the fundamental group
$G$ of a space of the following sort: one removes a family of
complex hyperplanes from $\C^n$, or complex hyperbolic space $\chn$, or the
Hermitian symmetric space for $\O(2,n)$, and then takes the quotient
by a discrete group $\PG$.  The classical example is the braid group,
but there are many similar ``braid-like'' groups that arise in
topology and algebraic geometry.  Our main result is that if $\PG$
contains reflections in the hyperplanes nearest the basepoint,
and these reflections satisfy a certain property, then $G$ is
generated by the analogues of the generators of the classical braid
group.  We apply this to obtain generators for $G$ in a particular
intricate example in $\ch^{13}$.  The interest in this example
comes from a conjectured relationship between this braid-like group and the
monster simple group~$M$, that gives geometric meaning to the
generators and relations in the Conway-Simons presentation of $(M\times M):2$.
\end{abstract}

\maketitle
%
%
\section{Introduction}
\label{sec-introduction}

\noindent
The usual $n$-strand braid group of the plane was described by Fox and
Neuwirth \cite{Fox-Neuwirth} as the fundamental group of $\C^n$, minus
the hyperplanes $x_i=x_j$, modulo the action of the group generated by
the reflections across them (the symmetric group $S_n$).  The idea is
that a path $(x_1(t),\dots,x_n(t))$ in this hyperplane complement
corresponds to the braid whose strands are
the graphs of the maps $t\mapsto x_i(t)$,
regarded as curves in $[0,1]\times\C$.  The removal of the hyperplanes corresponds to the condition
that the strands do not meet each other.  The fundamental group of the
hyperplane complement is the pure braid group.  Impure braids
correspond to paths, not loops, in the hyperplane complement.  
But these paths become loops once we quotient by $S_n$, and one obtains
the usual braid group.

The term ``braid-like'' in the title is meant to suggest
groups that arise by this construction, generalizing the
choices of $\C^n$ and this particular hyperplane arrangement.  Artin
groups \cite{Brieskorn} and the braid groups of finite complex reflection
groups \cite{Bessis} are examples.  The problem we address is: find
generators for groups of this sort.  We are mainly interested in  the case that there are infinitely many hyperplanes,
for example coming from hyperplane arrangements in complex hyperbolic
space $\chn$.  Our specific motivation is a conjecture relating the monster finite simple group to the braid-like group associated to a certain hyperplane
arrangement in $\ch^{13}$.  
By our results and those of Heckman \cite{Heckman} and
Heckman--Rieken
\cite{Heckman-Rieken}, this conjecture now seems approachable.
We also suggest some
applications to algebraic geometry.

The general setting is the following: Let $X$ be complex Euclidean
space, or complex hyperbolic space, or the Hermitian symmetric space
for an orthogonal group $\O(2,n)$.  Let $\M$ be a locally finite set
of complex hyperplanes in $X$, $\H$ their union, and $\PG\sset\Isom X$
a discrete group preserving $\H$.  Let $a\in X-\H$.  Then the
associated ``braid-like group'' means the orbifold fundamental group
$$
G_a:=\piorb\bigl((X-\H)/\PG,a\bigr)
$$
See section~\ref{sec-loops-in-quotients} for the precise definition of this.  In many cases, $\PG$ acts freely on $X-\H$, so that the
orbifold fundamental group is just the ordinary fundamental group.

Our first result describes generators for $\pi_1(X-\H,a)$.  This is a
subgroup of $G_a$, since $X-\H$ is an orbifold covering space of
$(X-\H)/\PG$.  For $H\in\M$ we define in
section~\ref{sec-loops-in-arrangement-complements} a loop $\Loop{a H}$
that travels from $a$ to a point $c\in X-\H$ very near $H$, encircles
$H$ once, and then returns from $c$ to $a$.  We pronounce the notation
``$a$ loop $H$'' or ``$a$ lasso $H$''.  See
section~\ref{sec-loops-in-arrangement-complements} for details and
a slight generalization (theorem~\ref{thm-generators-for-metric-neighborhood-of-A}) of the following result:

\begin{theorem}
\label{thm-generators-for-X-minus-hyperplanes}
The loops $\Loop{a H}$, with $H$ varying over $\M$, generate
$\pi_1(X-\H,a)$.\qed
\end{theorem}

If $a$ is generic enough then this follows easily from stratified
Morse theory \cite{Goresky-MacPherson}.  But in our applications it is
very important to take $a$ non-generic, because choosing it to have
large $\PG$-stabilizer can greatly simplify the analysis of
$\piorb\bigl((X-\H)/\PG,a\bigr)$.  So we prove
theorem~\ref{thm-generators-for-X-minus-hyperplanes} with no
genericity conditions on~$a$.  This lack of genericity complicates
even the definition of $\Loop{a H}$.  For example, $\Loop{a H}$ may
encircle some hyperplanes other than $H$, and this difficulty cannot
be avoided in any natural way (see remark~\ref{rem-non-generic-case}).  One might view
theorem~\ref{thm-generators-for-X-minus-hyperplanes} as a first step
toward a version of stratified Morse theory for non-generic
basepoints.

Next we consider generators for the orbifold fundamental group $G_a$
of $(X-\H)/\PG$.  In our motivating examples,
$\PG$ is generated by complex reflections in the hyperplanes $H\in\M$.
(A complex reflection means an isometry of finite order${}>1$ that
pointwise fixes a hyperplane, called its mirror.)  So suppose $H\in\M$
is the mirror of some complex reflection in $\PG$.  Then there is a
``best'' such reflection $R_H$, characterized by the following
properties: every complex
reflection in $\PG$ with mirror $H$ is a power of $R_H$, and $R_H$ acts on
the
normal bundle of $H$ by $\exp(2\pi i/n)$, where $n$ is the order
of~$R_H$.    

For each hyperplane $H\in\M$, we will define in
section~\ref{sec-loops-in-quotients} an element $\mu_{a,H}$ of $G_a$
which is the natural analogue of the standard generators for the
classical braid group.  Figure~\ref{fig-A2-example} illustrates the
construction for the $3$-strand braid group.  (We have drawn the 
subset of $\R^3\sset\C^3$ with coordinate sum~$0$, and our paths lie
in this $\R^2$ except where they dodge the hyperplanes.)  Recall that the definition of $\Loop{a H}$
referred to a point $c\in X-\H$ very near $H$, and a circle around $H$
based at~$c$.  We define $\mu_{a,H}$ to go from $a$ to $c$ as before,
then along the portion of this circle from $c$ to $R_H(c)$, then along
the $R_H$-image of the inverse of the path from $a$ to $c$.  This is a
path in $X-\H$, not a loop.  But $R_H$ sends its beginning point to
its end point, so we have specified a loop in $(X-\H)/\PG$.  So we may
regard $\mu_{a,H}$ as an element of $G_a$.  (Because $a$ may have
nontrivial stabilizer, properly speaking we must record the ordered
pair $(\mu_{a,H},R_H)$ rather than just $\mu_{a,H}$; see
section~\ref{sec-loops-in-quotients} for background on the orbifold
fundamental group.)

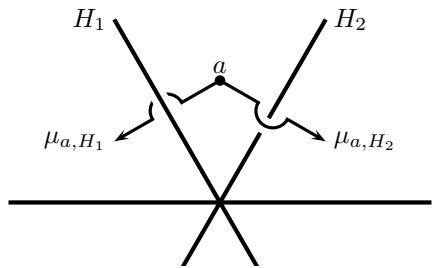
\begin{figure}
\psset{unit=80pt}
\def\URx{1}
\def\URy{.95}
\def\LLx{-\URx}
\def\LLy{-.3}
\begin{pspicture*}(\LLx,\LLy)(\URx,\URy)
\def\detourradius{.08}
\def\dotradius{.03}
\def\whitewidth{.08}
\def\hyperplanewidth{.02}
\def\pathwidth{.015}
\def\whitediskradius{.077} 
\psline[arrows=->,linewidth=\pathwidth](0,.577)(-.5,.289)
\pscircle[linestyle=none,fillstyle=solid,fillcolor=white](-.25,.433){\whitediskradius}
\psarc[linestyle=solid,fillstyle=none,linewidth=\pathwidth](-.25,.433){\detourradius}{30}{210}
\psline[linewidth=\whitewidth,linecolor=white](.5,-.866)(-.5,.866)
\psline[linewidth=\hyperplanewidth](.5,-.866)(-.5,.866)
\psline[arrows=->,linewidth=\pathwidth](0,.577)(.5,.289)
\pscircle[linestyle=none,fillstyle=solid,fillcolor=white](.25,.433){\whitediskradius}
\psline[linewidth=\hyperplanewidth](.5,.866)(-.5,-.866)
\psarc[linestyle=solid,fillstyle=none,linewidth=\whitewidth,linecolor=white](.25,.433){\detourradius}{210}{270}
\psarc[linestyle=solid,fillstyle=none,linewidth=\pathwidth](.25,.433){\detourradius}{150}{330}
\psline[linewidth=\hyperplanewidth](-1,0)(1,0) 
\pscircle[linestyle=none,fillstyle=solid,fillcolor=black](0,.577){\dotradius}
\uput{3pt}[90](0,.577){$a$}
\uput{3pt}[180](-.5,.866){$H_1$}
\uput{3pt}[180](-.5,.289){$\mu_{a,H_1}$}
\uput{3pt}[0](.5,.866){$H_2$}
\uput{3pt}[0](.5,.289){$\mu_{a,H_2}$}
\end{pspicture*}
\caption{The standard generators for the $3$-strand braid group, in
  terms of the $A_2$ hyperplane arrangement.}
\label{fig-A2-example}
\end{figure}

Referring again to figure~\ref{fig-A2-example}, generation of the $3$-strand braid
group $\piorb\bigl((\C^2-\H)/S_3,a\bigr)$ requires only the loops
$\mu_{a,H}$ for the hyperplanes $H$ closest to the basepoint $a$.  This is
proven in \cite{Fox-Neuwirth} for the $n$-strand braid
group.   For this to hold, one should choose $a$ as we did
here,  having the
same distance to every facet of the Weyl chamber.  
Our next theorem shows that the same holds in our more general
situation, provided that $a$ is chosen so that there are ``enough'' mirrors closest to it.
This analogy with the braid group  is the
source of the term ``geometric generators'' in our title.

\begin{theorem}
\label{thm-not-quite-strong-enough}
Suppose $\calC\sset\M$ are the hyperplanes closest to~$a$, and
that the complex reflections $R_C$ generate $\PG$, where $C$ varies
over $\calC$.  Suppose that for each $H\in\M-\calC$, some power of some
$R_C$ moves $a$ closer to $p$, where $p$ is the point of $H$ closest
to $a$.  Then the loops $(\mu_{a,C},R_C)$ generate
$G_a=\piorb\bigl((X-\H)/\PG,a\bigr)$.
\end{theorem}

The hypothesis of being able to move $a$ closer to the various $p$'s appears very
hard to check in practice.  But it can be done in some nontrivial
cases; its verification (slightly weakened) for our motivating example
occupies most of section~\ref{sec-Leech-meridians-generate}.

This motivating example is the setting for a conjectural relation between a
particular braid-like group and the monster simple group~$M$.  Here are
minimal details; see section~\ref{sec-background} for more background
and \cite{Allcock-monstrous} for full details.  We take $X$ to be complex hyperbolic
space $\ch^{13}$, and $\PG$ to be a particular finite-covolume 
discrete subgroup of
$\Aut\ch^{13}=PU(13,1)$, generated by complex reflections of order~$3$.
We take $\M$ to be the set of mirrors of the complex reflections in
$\PG$.  It turns out that any two mirrors are $\PG$-equivalent, so the
image of $\H$ in $X/\PG$ is irreducible.  The positively-oriented
boundary, of a small disk in $X/\PG$ transverse to a generic point of this image,
determines a conjugacy class in $\piorb\bigl((X-\H)/\PG\bigr)$.
Following knot theory terminology, we call the elements of this
conjugacy class {\it meridians}.

\begin{conjecture}[{\cite{Allcock-monstrous}}]
\label{conj-monstrous-proposal}
The quotient of $\piorb\bigl((X-\H)/\PG\bigr)$, by the normal subgroup
generated by the squares of the meridians, is the semidirect product
of  $M\times
M$ by $\Z/2$, where $M$ is the monster simple group and $\Z/2$
exchanges the factors in the obvious way.
\end{conjecture}

Presumably, any proof of this will require generators and relations
for $\piorb\bigl((X-\H)/\PG\bigr)$, which is the motivation for the
current paper.  In \cite{Basak-bimonster-1} the second author found a
point $\tau\in X-\H$ (called $\bar\rho$ there), such that the set
$\calC$ of mirrors closest to~$\tau$ has size~$26$, and showed that
their complex reflections generate $\PG$.  Describing~$\tau$ precisely
requires some preparation, so we refer to section~\ref{sec-background} for details
and just
mention that~$\tau$
has nontrivial $\PG$-stabilizer.  Therefore the corresponding meridians are ordered pairs
$(\mu_{\tau,C},R_C)$ rather than just bare paths $\mu_{\tau,C}$.
Taking~$\tau$ as our basepoint,
we announce the following result, which we regard as a significant
step toward conjecture~\ref{conj-monstrous-proposal}.  

\begin{theorem}
\label{thm-announcement}
The $26$ meridians $(\mu_{\tau,C},R_C)$, with $C$ varying over the $26$ mirrors
of $\M$ closest to $\tau$, generate $\piorb\bigl((X-\H)/\PG,\tau\bigr)$.
\qed
\end{theorem}

We wish this were a corollary of
theorem~\ref{thm-not-quite-strong-enough}.  Unfortunately the
hypothesis there about being able to move $\tau$ closer to the various
$p\in H\in\M$ fails badly.  Instead, the proof goes as follows.
First, in section~\ref{sec-Leech-meridians-generate} we prove theorem~\ref{thm-leech-generation}
below, which is the analogue of theorem~\ref{thm-announcement} with a
different basepoint $\rho$ in place of $\tau$.  For this basepoint,
one can (almost) verify the hypothesis of theorem~\ref{thm-not-quite-strong-enough} about being
able to move $\rho$ closer to the various $p$'s.  Considerable
calculation is required, together with extra work dealing with the fact that
this hypothesis almost holds but not quite.

Then, given theorem~\ref{thm-leech-generation}, one joins $\tau$ and $\rho$ by a path and uses it
to identify the fundamental groups based at these points.  One can
then prove theorem~\ref{thm-announcement} by studying how these groups' generators are
related.  The argument is delicate, of more specialized interest, and
has little in common with the ideas in this paper.  Therefore it will
appear separately.

After stating theorem~\ref{thm-leech-generation} we will explain our real reason for
preferring $\tau$ over $\rho$ as a basepoint.  A minor additional
reason is that $\rho$ is not a point of $\ch^{13}$.
Rather, it is a cusp of $\PG$, and in particular lies in the sphere at
infinity $\partial\ch^{13}$.  This complicates things in two ways.
First, there are infinitely many mirrors ``closest'' to $\rho$,
indexed by the elements of a certain $25$-dimensional integral
Heisenberg group.  And second, the definition of the meridians ``based
at~$\rho$'' requires more care.  This leads to theorem~\ref{thm-leech-generation} giving
an infinite generating set, consisting of paths that are
more complicated than
those of theorem~\ref{thm-announcement}.

One can work through these complications as follows.  As we did for
$\tau$, we refer to section~\ref{sec-background} for the precise
definition of $\rho$.
All we need for now is that it is a cusp of $\PG$ and that there is a
closed horoball $A$ centered at $\rho$ that misses $\H$
(lemma~\ref{lem-disjoint-horoball-exists}).
We choose any basepoint $a$ inside~$A$.  We call the
mirrors that come closest to $A$ the ``Leech mirrors''.  The name
comes from the fact that they are indexed by the elements of (a
central extension of) the complex Leech lattice $\Leech$; in
particular there are infinitely many of them.  If $C$ is a Leech
mirror, let $b_C\in A$ be the point of $A$ nearest it.  Then
$\mu_{a,A,C}$ is defined to be the geodesic $\geodesic{a b_C}\sset A$
followed by $\mu_{b_C,C}$ followed by $R_C(\geodesic{b_C a})$.  See
figure~\ref{fig-illustration-of-mu} for a picture.  These are
meridians in the sense of conjecture~\ref{conj-monstrous-proposal},
and we call them the Leech meridians.  As before, because $a$ may
have nontrivial $\PG$-stabilizer, the meridian associated to $C$ is
really the ordered pair $(\mu_{a,A,C},R_C)$ rather than just
the bare path $\mu_{a,A,C}$.

\begin{theorem}
\label{thm-leech-generation}
The orbifold fundamental group $\piorb\bigl((X-\H)/\PG,a\bigr)$ is generated by the Leech meridians,
that is, by the loops $(\mu_{a,A,C},R_C)$ with $C$ varying over the
(infinitely many) mirrors closest to $\rho$.
\end{theorem}

We promised to explain the real reason we prefer theorem~\ref{thm-announcement} to
theorem~\ref{thm-leech-generation}, that is, why we prefer the basepoint to
be $\tau$ rather than $\rho$.  It is because the $26$ meridians
of theorem~\ref{thm-announcement} are closely related to the coincidences that motivated
conjecture~\ref{conj-monstrous-proposal}.  In particular, by
\cite{Basak-bimonster-2} they satisfy the braid and commutation
relations specified by the incidence graph of the $13$ points and $13$
lines of $P^2(\F_3)$.  That is, two generators $x,y$ commute ($x y=y x$)
or braid ($x y x=y x y$) according to whether the corresponding nodes of
this graph are unjoined or joined.  We call the abstract group with
$26$ generators, subject to these relations, the Artin group of
$P^2(\F_3)$.

There are a family of presentations of the ``bimonster'' $(M\times
M)\rtimes \Z/2$ as a quotient of this Artin group, due to Conway,
Ivanov, Norton, Simons and Soicher in various combinations.  All of
them impose the relations that the generators have order~$2$, which
yields the (infinite) Coxeter group whose Coxeter diagram is the same
incidence graph.  Modding out by the squares of meridians in
conjecture~\ref{conj-monstrous-proposal} corresponds to taking this quotient.  There are
several different ways to specify additional relations that collapse
this Coxeter group to the bimonster.  The most natural one for our
purposes seems to be the ``$\Atilde_{11}$ deflation'' relations of
\cite{Conway-Simons}, because these have a good geometric
interpretation in terms of the $\mu_{\tau,C}$'s and a certain
copy of $\ch^9$ in $\ch^{13}$.  See  \cite{Heckman} and
\cite{Heckman-Rieken} for more details.

\medskip
We hope that our techniques will be useful more generally.  For
example, they might be used to give generators for the fundamental
group of the moduli space of Enriques surfaces.  Briefly, this is the
quotient of the Hermitian symmetric space for $\O(2,10)$, minus a
hyperplane arrangement, by a certain discrete group.  See
\cite{Namikawa} for the original result and
\cite{Allcock-Enriques} for a simpler description of the arrangement.
The symmetric space has two orbits of $1$-dimensional cusps, one of
which misses all the hyperplanes.  Taking this as the base
``point'', the hyperplanes nearest it are analogues of the Leech
mirrors.
%
It seems reasonable to hope that the meridians associated to these mirrors generate the
orbifold fundamental group.

There are many spaces in algebraic geometry with a description
$(X-\H)/\PG$ of the sort we have studied.  For example, the
discriminant complements of many hypersurface singularities
\cite{Looijenga-elliptic-singularities}\discretionary{}{}{}\cite{Looijenga-triangle-singularities},
the moduli spaces of del Pezzo surfaces
\cite{ACT-surfaces}\discretionary{}{}{}\cite{Kondo-genus-3}\discretionary{}{}{}\cite{Heckman-Looijenga}, 
the moduli space of curves of genus four \cite{Kondo-genus-4},
the moduli
spaces of smooth cubic threefolds
\cite{ACT-threefolds}\discretionary{}{}{}\cite{Looijenga-Swierstra} and fourfolds
\cite{Looijenga-fourfolds}, and the moduli spaces of lattice-polarized
K3 surfaces \cite{Nikulin}\discretionary{}{}{}\cite{Dolgachev}.  The orbifold fundamental
groups of these spaces are ``braid-like'' in the sense of this paper,
and we hope that our methods will be useful in understanding them.

\medskip
The paper is organized as follows.  In section~\ref{sec-loops-in-arrangement-complements} we study the
fundamental group $\pi_1(X-\H,a)$, in particular proving
theorem~\ref{thm-generators-for-X-minus-hyperplanes}.  The proof relies on van Kampen's theorem.  In
section~\ref{sec-loops-in-quotients} we study $\piorb\bigl((X-\H)/\PG,a\bigr)$, in
particular proving theorem~\ref{thm-not-quite-strong-enough}.  The core of that proof is
lemma~\ref{lem-lowering-lemma}, which is more general than needed for theorem~\ref{thm-not-quite-strong-enough}.
The extra generality is needed for our application to $\ch^{13}$.
Section~\ref{sec-background} gives background on complex hyperbolic space and the
particular hyperplane arrangement referred to in conjecture~\ref{conj-monstrous-proposal} and
theorems~\ref{thm-announcement}--\ref{thm-leech-generation}.  Finally, section~\ref{sec-Leech-meridians-generate} proves
theorem~\ref{thm-leech-generation}.  Most of the proof consists of tricky calculations
verifying the hypothesis of theorem~\ref{thm-not-quite-strong-enough} that the basepoint can be
moved closer to the various points $p\in H$.  In a few cases this is
not possible, so we have to do additional work.

The first author is very grateful to RIMS at Kyoto University for its
hospitality during two extended visits while working on this paper.
The second author would like to thank Kavli-IPMU at University of Tokyo 
for their hospitality during two one-month visits while working on this paper. 

\section{Loops in arrangement complements}
\label{sec-loops-in-arrangement-complements}
%
%
\noindent
For the rest of the paper we fix $X={}$one of three spaces, $\M={}$a
locally finite set of hyperplanes in $X$, and $\H={}$their union.  The
precise assumption on $X$ is that it is complex affine space with its
Euclidean metric, or complex hyperbolic space, or the Hermitian symmetric space
for $\O(2,n)$.  To understand the general machinery in this section
and the next, it is enough to think about the affine case.  In our
application in section~\ref{sec-Leech-meridians-generate} we
specialize to the case that $X$ is complex hyperbolic $13$-space;
for background see section~\ref{sec-background}.
Most of the other potential applications mentioned in the introduction
would use the $\O(2,n)$ case.

We will freely use a few standard properties of $X$: it is
contractible, and its natural Riemannian metric is complete and has
nonpositive sectional curvature.  By \cite[Theorems~II.1A.6 and
  II.4.1]{Bridson-Haefliger}, another way to state these properties is
that $X$ is a complete CAT(0) space.  We will also use the usual
notions of  
complex lines and complex hyperplanes in $X$, both of which are
totally geodesic.

For $b,c\in X$ we write $\geodesic{b c}$ for the
geodesic segment from $b$ to $c$.  Now suppose $b,c\notin\H$.  It may happen
that $\geodesic{b c}$ meets $\H$, so we will define a perturbation
$\dodge{b c}$ of $\geodesic{b c}$ in the obvious way.  The notation may be
pronounced ``$b$ dodge $c$'' or ``$b$ detour $c$''.  
We write $\complexgeodesic{b c}$ for the complex line containing
$\geodesic{b c}$.  By the local finiteness of $\M$, $\complexgeodesic{b c}\cap\H$ is a discrete
set.
Consider the path got
from $\geodesic{b c}$ by using positively oriented semicircular detours
in $\complexgeodesic{b c}$, around the points of $\geodesic{b c}\cap\H$,
in place of the corresponding segments of $\geodesic{b c}$.  After taking
the radius of these detours small enough, the construction
makes sense and the resulting homotopy class in $X-\H$ (rel
endpoints) is radius-independent.  This homotopy class is what we mean
by $\dodge{b c}$.

At times we will need to speak of the ``restriction of $\H$ at
$p$'', where $p$ is a point of~$X$.  So we write $\M_p$ for the set of hyperplanes in $\M$ that
contain~$p$, and $\H_p$ for their union.

Now suppose $b\in X-\H$ and $H\in\M$.  We will define a homotopy class
$\Loop{b H}$ of loops in $X-\H$ based at~$b$; the notation can be
pronounced ``$b$ loop~$H$'' or ``$b$ lasso~$H$''.  We write $p$ for
the point of $H$ nearest $b$.  It exists and is unique by 
$H$'s convexity and $X$'s nonpositive curvature
\cite[Prop.\ II.2.4]{Bridson-Haefliger}.  
Let $U$ be a ball around $p$ that is
small enough that $U\cap\H=U\cap\H_p$, and let $c$ be a point of
$(\geodesic{b p}\cap U)-\{p\}$.  Consider the circular loop in
$\complexgeodesic{b p}$ centered at $p$, based at $c$, and traveling
once around $p$ in the positive direction.  It misses $\H$, because
under the exponential map $T_p X\to X$, the elements of $\M_p$
correspond to complex hyperplanes in $T_p X$, while
$\complexgeodesic{b p}$ corresponds to a complex line.  And the line
misses the hyperplanes except at~$0$, because $b\notin\H$.  Finally,
$\Loop{b H}$ means $\dodge{b c}$ followed by this circular loop,
followed by $\reverse(\,\dodge{b c}\,)$.

\begin{remark}[Caution in the non-generic case]
\label{rem-non-generic-case}
This definition has some possibly unexpected behavior when $b$ is not
generic.  For example, take $\M$ to be the  $A_2$ arrangement
in $X=\C^2$, let $H$ be one of the three hyperplanes, and take $b\in X-\H$
orthogonal to~$H$.  
See figure~\ref{fig-A2-example}, with that figure's $a$ being $b$, and the unlabeled
hyperplane being~$H$.
It is easy to see that $\Loop{b H}$ encircles {\it
  all three hyperplanes}, not just~$H$.  Furthermore, this phenomenon cannot be avoided
by any procedure that respects symmetry.  To explain this we 
note that $\pi_1(X-\H,b)\iso \Z\times F_2$, where the first factor
is generated by $\Loop{b H}$ and the second is free on
$\Loop{b H_1}$ and $\Loop{b H_2}$, and $H_1$ and $H_2$
are the other two hyperplanes.  Let $f$ be the isometry of $X$ that fixes
$b$ and negates $H$.  It exchanges $H_1$ and $H_2$.  
So $f$'s action on $\pi_1(X-\H,b)$ fixes $\Loop{b H}$ and swaps the other
two generators.  It follows that the group of fixed points of $f$ in
$\pi_1(X-\H,b)$ is just the first factor~$\Z$.  So any
symmetry-respecting definition of $\Loop{b H}$ must give some power
of our definition.
\end{remark}

The main result of this section,
theorem~\ref{thm-generators-for-metric-neighborhood-of-A}, shows that
the various $\Loop{b H}$ generate $\pi_1(X-\H,b)$.  But for our
applications to $\ch^{13}$ in
section~\ref{sec-Leech-meridians-generate}, it will be useful to
formulate the fundamental group with a ``fat basepoint'' $A$ in place
of $b$.  This is because we will want to choose our basepoint to be a
cusp of a finite-covolume discrete subgroup of
$\PU(13,1)=\Aut\ch^{13}$.  Strictly speaking this is not possible,
since a cusp is not a point of $\ch^{13}$.  So we will use a closed
horoball $A$ centered at that boundary point in place of a basepoint.
For purposes of understanding the current section, the reader may take
$A$ to be a point.

\medskip
Our assumptions so far are that $X$ is one of three spaces, $\M$ is a
locally finite hyperplane arrangement, and $\H$ is the union of the
hyperplanes.  Henceforth we also assume that $A$ is a nonempty closed
convex subset of $X$, disjoint from $\H$.  To avoid some minor
technical issues, we assume two more properties, both
automatic when $A$ is a point.  First, for every
$H\in\M$, there is a unique point of $A$ closest to $H$.  (This holds
if $A$ is strictly convex, by the argument used for
\cite[Prop.\ II.2.4]{Bridson-Haefliger}.)  Second, some group of
isometries of $X$, preserving $\M$ and $A$, acts cocompactly on the
boundary $\partial A$.  (This holds in our application to $\ch^{13}$
because the stabilizer of a cusp acts cocompactly on any horosphere
centered there.)

We will think of the open $r$-neighborhood $B_r$ of
$A$, for some $r>0$ as being ``like'' an open ball.  If $A$ is a point
then of course $B_r$ actually is one.  In any case, the convexity
of $A$ implies that of $B_r$ by \cite[Cor.\ II.2.5]{Bridson-Haefliger}
and the remark preceding it.

Because $A-\H=A$ is simply connected (even contractible), the
fundamental groups of $X-\H$ based at any two points of $A$ are
canonically identified.  So we write just $\pi_1(X-\H,A)$
for $\pi_1(X-\H,a)$, where $a$ is any point of $A$.  If $c\in
X$ then we define $\geodesic{Ac}$ 
as $\geodesic{b c}$, where $b$ is the point of $A$ nearest $c$.
If $c\not\in\H$ then we also define $\dodge{Ac}$ 
as $\dodge{b c}$.
Similarly, if $H\in\M$ then we define $\Loop{AH}$ as $\Loop{b H}$,
where $b$ is the point of $A$ closest
to $H$.  
We sometimes  write $\geodesic{b,c}$ and $\dodge{b,c}$ and $\Loop{b,H}$
for  $\geodesic{b c}$ and $\dodge{b c}$ and $\Loop{b H}$,
and similarly for $\geodesic{A,c}$ and $\dodge{A,c}$ and $\Loop{A,H}$.

\begin{theorem}[$\pi_1$ of a ball-like set minus hyperplanes]
\label{thm-generators-for-metric-neighborhood-of-A}
Let $B_r$ be the open $r$-neighborhood of $A$, where $r\in(0,\infty]$.
  Then $\pi_1(B_r-\H,A)$ is generated by the $\Loop{AH}$'s for which
  $d(A,H)<r$.
\end{theorem}

If $A$ is compact, for example $A=\{a\}$, then this gives a finite
number of generators for $\pi_1(B_r-\H)$.  But if $A$ is non-compact
then the number of generators may be infinite.   This
happens in theorem~\ref{thm-leech-generation}, where $A$ is a horoball in $\ch^{13}$.  The
rest of the section is devoted to the proof of theorem~\ref{thm-generators-for-metric-neighborhood-of-A}, beginning with two
lemmas.

\begin{lemma}[$\pi_1$ of  $\C^n$ minus hyperplanes
    through the origin]
\label{lem-pi-1-is-product}
Suppose $X$ is complex Euclidean space, every $H\in\M$ contains the
origin~$0$, and $c\in X-\H$.  
Write $\halfX$ for the open halfspace of $X$
that contains $c$ and is bounded by the real orthogonal complement
to~$\geodesic{c0}$.  $($In the trivial case $\M=\emptyset$
we also assume $c\neq0$, so that $\halfX$ is defined.$)$  
\begin{enumerate}
\item
\label{item-if-H-not-in-M}
If $c$ is not orthogonal to any element of $\M$, then
$\pi_1(X-\H,c)$ is generated by $\pi_1(\halfX-\H,c)$. 
\item
\label{item-if-H-in-M}
If $c$ is orthogonal to some $H\in\M$, then $\pi_1(X-\H,c)$ is generated by 
$\pi_1(\halfX-\H,c)$ together with any element of $\pi_1(X-\H,c)$ having
linking number $\pm1$ with $H$, for example $\Loop{c H}$.
\end{enumerate}
\end{lemma}

\begin{proof}
\eqref{item-if-H-in-M} Write $H'$ for the
translate of $H$ containing $c$.  Every point of $X-\H$ is a nonzero
scalar multiple of a unique point of $H'-\H$.  It follows that $X-\H$
is the topological product of $H'-\H\sset\halfX-\H$ with $\C-\{0\}$.  The map
$\pi_1(X-\H,c)\to\Z$ corresponding to the projection to the second
factor is the linking number with $H$.  All that remains to prove is
that $\Loop{c H}$ has linking number~$1$ with $H$.  In fact more is
true: essentially by definition, this loop generates the fundamental
group of the factor $\C-\{0\}$.

\eqref{item-if-H-not-in-M} 
We define $H$ as the complex hyperplane through~$0$ that is orthogonal
to $\geodesic{c0}$. 
We apply the previous paragraph to
$\M'=\M\cup\{H\}$ and $\H'=\H\cup H$. 
Using $\halfX-\H=\halfX-\H'$
yields 
\begin{equation}
\label{eq-direct-product-decomposition}
\pi_1(X-\H',c)
=
\pi_1(\halfX-\H',c)\times\biggend{\Loop{c H}}
=
\pi_1(\halfX-\H,c)\times\biggend{\Loop{c H}}.
\tag{*}
\end{equation}
Our goal is to show that the first factor on the right surjects to
$\pi_1(X-\H,c)$. 
Let $\gamma$ be any element of $\pi_1(X-\H',c)$ that is freely
homotopic to 
the boundary of a small disk transverse to $H$ at a
generic point of $H$.  It dies under the natural map
$\pi_1(X-\H',c)\to\pi_1(X-\H,c)$.  
Because
$\gamma$ has linking number $\pm1$ with $H$, the  product
decomposition \eqref{eq-direct-product-decomposition}
shows that every element of $\pi_1(X-\H',c)$ can be
written as  a power of $\gamma$ times an element of
$\pi_1(\halfX-\H,c)$.  It is standard that
$\pi_1(X-\H',c)\to\pi_1(X-\H,c)$ is surjective.  (Take any loop in
$X-\H$, perturb it to miss
$H$,  and then regard it as a loop in $X-\H'$.)  Since this map kills
$\gamma$, it must send the subgroup $\pi_1(\halfX-\H,c)$ of $\pi_1(X-\H',c)$ 
surjectively to $\pi_1(X-\H,c)$.
\end{proof}

\begin{lemma}[$\pi_1$ of a ball-like set with a bump, minus hyperplanes]
\label{lem-generators-for-pi-1-of-ball-plus-small-bulge}
Suppose $r>0$, $B$ is the open $r$-neighborhood of $A$, and
$p\in\partial B$.  Assume $U$ is any open ball centered at $p$, small
enough that $U\cap\H=U\cap\H_p$.
\begin{enumerate}
\item
\label{case-union-when-H-not-a-mirror}
If no $H\in\M_p$ is orthogonal to $\geodesic{A p}$, then 
$\pi_1\bigl((B\cup U)-\H,A\bigr)$ is generated by
the image of $\pi_1(B-\H,A)$.  
\item
\label{case-union-when-H-a-mirror}
If some $H\in\M_p$ is orthogonal to $\geodesic{A p}$, then
$\pi_1\bigl((B\cup U)-\H,A\bigr)$ is generated by the image of
$\pi_1(B-\H,A)$, together with any loop 
of the following
form $\alpha\lambda\alpha^{-1}$, for example $\Loop{AH}$.  Here $\alpha$ is a path in $B-\H$ from
$A$ to a point of $(B\cap U)-\H$ and $\lambda$ is a
loop in $U-\H$, based at that point and  having linking number~$\pm1$ with $H$.
\end{enumerate}
\end{lemma}

\begin{proof}
For uniformity, in case \eqref{case-union-when-H-not-a-mirror} we
choose some path $\alpha$ in $B-\H$ beginning in $A$ and ending in
$(B\cap U)-\H$.  In both cases we write $c$ for the
final endpoint of $\alpha$; without loss of generality we may suppose $c\in\geodesic{A p}-\{p\}$.
Van Kampen's theorem shows that
$\pi_1\bigl((B\cup U)-\H,c\bigr)$ is generated by the images of
$\pi_1(B-\H,c)$ and $\pi_1(U-\H,c)$.  We claim that $\pi_1(U-\H,c)$
is generated by the image of $\pi_1\bigl((B\cap U)-\H,c)$,
supplemented in case~\eqref{case-union-when-H-a-mirror} by
$\lambda$.  

Assuming this, we move the basepoint from $c$ into $A$ along
$\reverse(\alpha)$.  This identifies 
$\pi_1(B-\H,c)$ with $\pi_1(B-\H,A)$, $\lambda$ with $\alpha\lambda\alpha^{-1}$, 
and
the elements of $\pi_1\bigl((B\cap U)-\H,c\bigr)$ with
certain loops in $B-\H$ based in $A$.  It follows that
$\pi_1\bigl((B\cup U)-\H,A\bigr)$ is generated by the image of
$\pi_1(B-\H,A)$, supplemented in
case~\eqref{case-union-when-H-a-mirror} by $\alpha\lambda\alpha^{-1}$.  This is the
statement of the theorem.

So it suffices to prove the claim.  We transfer this to a problem in
the tangent space $T_p X$ by the exponential map and its inverse
(written $\log$).  
So we must show that
$\pi_1(\log U-\log\H_p,\log c)$ is generated by the image of
$\pi_1\bigl(\log (B\cap U)-\log\H_p,\log c\bigr)$, supplemented in
case~\eqref{case-union-when-H-a-mirror} by $\log\lambda$.  The key to this is that the vertical
arrows in the following commutative diagram are homotopy equivalences.
$$
\begin{CD}
\log(B\cap U)-\log\H_p
@>>>
\log U-\log\H_p
\\
@VVV
@VVV
\\
\halfTpX-\log\H_p
@>>>
T_p X-\log\H_p
\\
\end{CD}
$$
Here $\frac{1}{2}T_p X$ is as in
lemma~\ref{lem-pi-1-is-product}: the open halfspace containing
$\log c$ and bounded by the (real) orthogonal complement of
$\log(\geodesic{c p})=\geodesic{\log c,0}$.  

The right vertical arrow is a weak homotopy equivalence by a standard
scaling argument:  any compact set in $T_p X$ can be scaled down
until it lies in $\log U$, and scaling preserves $\log\H_p$.  

The same argument works for the left one: since the boundary of $\log
B$ is tangent to the boundary of $\frac{1}{2}T_p X$ at~$0$, it is easy
to see that  any compact subset of $\frac{1}{2}T_p X$ can be
sent into $\log B\cap \log U$ by multiplying it by a sufficiently small
scalar.  (One might worry about basic properties of $\partial B$ like
smoothness, since $B$ was defined in terms of $A$, on which we made no
smoothness assumptions.  One can circumvent all worry by observing the
following consequence of the triangle inequality: for small 
$\e$, $B$ contains the open $\e$-ball $D$ around the point of
$\geodesic{A p}$ at distance $\e$ from $p$. And the boundary of $\log
D$ is indeed smooth and tangent to $\frac12 T_p X$.)

Both weak homotopy equivalences are homotopy equivalences by
Whitehead's theorem (or one could refine the scaling argument).  To
prove the theorem it now suffices to show that $\pi_1(T_p
X-\log\H_p,\log c)$ is generated by the image of $\pi_1(\frac{1}{2}T_p
X-\log\H_p,\log c)$, supplemented in
case~\eqref{case-union-when-H-a-mirror} by $\log\lambda$.  This is
just lemma~\ref{lem-pi-1-is-product}, completing the proof.
\end{proof}

\begin{proof}[Proof of theorem~\ref{thm-generators-for-metric-neighborhood-of-A}]
Let $R$ be the set of $r\in(0,\infty]$ for which the conclusion of the
  theorem holds.  By our assumption that some group of isometries of
  $X$ preserves $\M$ and $A$ and acts cocompactly on $\partial A$, the
  distances $d(A,H)$ are bounded away from~$0$, as $H$ varies over
  $\M$.  Therefore $B_r\cap\H=\emptyset$ for all sufficiently small
  $r$.  It follows that $R$ contains all
small enough $r$.
We will show below that that if $r\in R-\{\infty\}$ then
  $[r,r+\d)\sset R$ for some $\d>0$.  Also, we obviously have
  $B_r=\cup_{q<r}\,B_q$ for any $r\in(0,\infty]$.  Therefore $(0,r)\sset
  R$ implies $(0,r]\sset R$.  The connectedness of $(0,\infty]$ then
      implies $R=(0,\infty]$, proving the theorem.

So fix $r\in R-\{\infty\}$; we will exhibit $\delta>0$ such that
$[r,r+\delta)\sset R$. By $r\in R$ we know that $\pi_1(B_r-\H,A)$ is
generated by the $\Loop{A H}$'s for which $d(A,H)<r$.  We abbreviate
$B_r$ to $B$ and define $S$ as the ``sphere'' $\partial B$.  For each
$p\in S$ there is an open ball $U_p$ centered at $p$ such that
$U_p\cap\H=U_p\cap\H_p$.  The cocompact action on $\partial A$ we used
in the previous paragraph is also cocompact on $S$.  Also, we may choose
the balls $U_p$ so that the set of all of them is preserved by this
action.  It follows that there exists $\d>0$ such that $B_{r+\d}$ is
covered by $B$ and all the $U_p$'s.

To prove $[r,r+\d)\sset R$ we suppose given some $r'\in(r,r+\d)$ and
  write $B'$ for $B_{r'}$.  Since $B'$ is covered by $B$ and the
  $U_p$'s, every mirror that meets $B'$ either meets $B$ or is tangent
  to $S$.  So we must prove that $\pi_1(B'-\H,A)$ is generated by
  $\pi_1(B-\H,A)$ and the $\Loop{AH}$'s with $H\in\M$ tangent to~$S$.
  Lemma~\ref{lem-generators-for-pi-1-of-ball-plus-small-bulge} says
  that $\pi_1\bigl((B\cup U_p)-\H,A\bigr)$ is generated by
  $\pi_1(B-\H,A)$, supplemented by $\Loop{AH}$ if $p$ is the point of
  tangency of $S$ with some $H\in\M$.  

For $p\in S$ we define $V_p=(B\cup U_p)\cap B'$.  It is easy to
see that the inclusion $ V_p-\H\to(B\cup U_p)-\H$ is a homotopy equivalence.
(Retract points of $U_p-B'$ along geodesics toward~$p$.)  So 
$\pi_1(V_p-\H,A)$ is generated by
  $\pi_1(B-\H,A)$, supplemented by $\Loop{AH}$ if $p$ is the point of
  tangency of $S$ with some $H\in\M$.  
Because $B'=\cup_{p\in S}V_p$, 
repeatedly using van Kampen's theorem shows that $\pi_1(B'-\H,A)$ is
generated by the images therein of all the $\pi_1(V_p-\H,A)$,
finishing the proof.  

This use of van Kampen's theorem requires checking that every set got
from the $V_p$'s by repeated unions and intersections is connected.  To
help verify this, we call a subset $Y$ of $X$ star-shaped (around $A$)
if it contains $A$ and the geodesics $\geodesic{A y}$ for all $y\in
Y$.  The lemma below shows that each $B\cup U_p$ is star-shaped.
Intersecting with $B'$ preserves star-shapedness and yields~$V_p$.
Since unions and intersections of star-shaped sets are again
star-shaped, our repeated application of van Kampen's theorem is
legitimate.
\end{proof}

\begin{lemma}
\label{lem-CAT-0}
In the notation of the previous proof, $B\cup U_p$ is star-shaped
around~$A$.
\end{lemma}

\begin{proof}
We must show that $y\in U_p$ implies $\geodesic{A y}\sset B\cup U_p$.
It suffices to prove $\geodesic{y z}\sset U_p$, where $z$ is the point
of $\partial B$ closest to $y$.  (We remarked above that the convexity of
$A$ implies that of $B$, and the uniqueness of  $z$ then follows from
\cite[Prop.\ II.2.4]{Bridson-Haefliger}.)  Note that $\geodesic{z p}$
lies in the closure of $B$, since $z$ and $p$ do and $B$ (hence its
closure) is convex.

Consider the triangle $p,y,z$ in $X$ and a comparison triangle,
meaning a triangle $p',y',z'$ in $\R^2$ with the same edge lengths.
We write $\theta$ for the angle between $\geodesic{z y}$ and
$\geodesic{z p}$ at $z$, and similarly for $\theta'$.  
Since $X$ is a CAT(0) metric space 
we have
$\theta'\geq\theta$ by \cite[Prop.~II.3.1]{Bridson-Haefliger}.  And we have
$\theta\geq\pi/2$, because otherwise we could shorten $\geodesic{y z}$
by moving $z$ toward $p$.  
Therefore $\theta'$ is the largest
angle of the comparison triangle, so $\geodesic{p' y'}$ is its longest
edge.  Since the two triangles have the same edge lengths,
$\geodesic{p z}$ is shorter than $\geodesic{p y}$, so 
$z\in U_p$.  Then $U_p$'s
convexity shows that all of $\geodesic{y z}$ lies in $U_p$.
\end{proof}

\section{Loops in quotients of arrangement complements}
\label{sec-loops-in-quotients}

\noindent
We continue using the notation $X$, $\M$ and $\H$ from the previous section.
We also suppose a group $\PG$ acts  isometrically and
properly discontinuously on $X$, preserving $\H$.  At this point we
have no group $\Gamma$ in mind; the notation $\PG$ is just for
compatibility with sections~\ref{sec-background}--\ref{sec-Leech-meridians-generate}.  Our goal is to understand the
orbifold fundamental group of $(X-\H)/\PG$.  We use the following
definition from \cite{Looijenga-Artin-groups} and \cite{Basak-bimonster-2}; more general formulations
exist \cite{Ratcliffe}\cite{Kapovich}.

Fixing a basepoint $a\in X-\H$, consider the set of
pairs $(\gamma,g)$ where $g\in\PG$ and $\gamma$ is a path in $X-\H$
from $a$ to $g(a)$.  We regard one such pair as equivalent to another
one $(\gamma',g')$ if $g=g'$ and $\gamma$ and $\gamma'$ are homotopic
in $X-\H$, rel endpoints.   The orbifold
fundamental group
$
G_a:=\piorb\bigl((X-\H)/\PG,a\bigr)
$
means the set of equivalence classes. 
The group operation
is 
$(\gamma,g)\cdot(\gamma',g')=(\hbox{$\gamma$ followed by
  $g\circ\gamma'$},g g')$.
Projection of $(\gamma,g)$ to $g$ defines a homomorphism $G_a\to\PG$.
It is surjective because $X-\H$ is connected.
The kernel is obviously $\pi_1(X-\H,a)$, yielding the exact sequence
\begin{equation}
\label{eq-exact-sequence-on-pi-1}
1\to\pi_1(X-\H,a)\to G_a\to\PG\to1
\end{equation}

Although we don't need it, we remark that if $a$ has trivial $\PG$
stabilizer then there is a simpler $\PG$-invariant description of the
orbifold fundamental group.  Writing $o$ for $a$'s orbit, we define
$G_o:=\piorb\bigl((X-\H)/\PG,o\bigr)$ as the set of $\PG$-orbits on
the homotopy classes (rel endpoints) of paths in $X-\H$ that begin and
end in $o$.  The $\PG$-action is the obvious one: $g\in\PG$ sends a
path $\gamma$ to~$g\circ\gamma$.  To define $\gamma\gamma'$, where
$\gamma,\gamma'\in G_o$, one translates $\gamma'$ so that it begins
where $\gamma$ ends, and then composes paths in the usual way.
Well-definedness of multiplication, and the identification with the
definition of $G_a$, uses the fact that every path starting in $o$ has
a unique translate starting at $a$.

A complex reflection means a finite-order isometry of $X$ whose
fixed-point set is a complex hyperplane, called its mirror.  In our
applications, $\PG$ is generated by complex reflections whose mirrors
are hyperplanes in~$\M$.  This leads to certain natural elements of
the orbifold fundamental group: for $H\in\M$ we next define a loop
$\mu_{a,H}\in G_a$ which is a fractional power of $\Loop{a H}$.  
(The meridians of conjecture~\ref{conj-monstrous-proposal} and theorems~\ref{thm-announcement}--\ref{thm-leech-generation} are
special cases of these loops.)
Write
$p$ for the point of $H$ closest to~$a$ and $n_H$ for the order of the
cyclic group generated by the complex reflections in $\PG$ with mirror
$H$.  Write $R_H$ for the isometry of $X$ that fixes $H$ pointwise and
acts on its normal bundle by $\exp(2\pi i/n_H)$.  This is a complex
reflection and lies in $\PG$, except when $H$ is not the mirror of any
 reflection in $\PG$.  In that case, $R_H$ is the identity map.

Recall that the definition of $\Loop{a H}$ involved a point $c$ of
$\geodesic{a p}$ very near $p$, and a circular loop in
$\complexgeodesic{a p}$ centered at $p$ and based at~$c$.  We define
$\mu_{a,H}$ as $\dodge{a c}$ followed by the first $(1/n_H)$th of this
loop (going from $c$ to $R_H(c)$), followed by $R_H(\reverse(\dodge{a
  c}))$.  (One can see such a path in figure~\ref{fig-illustration-of-mu}, although the
notation there is intended for a more elaborate situation considered
below.  The portion of the path in the figure that goes from $b$ to $R_H(b)$
is $\mu_{b,H}$.)
This is a path from $a$ to $R_H(a)$, so the pair $(\mu_{a,H},R_H)$ is
an element of the orbifold fundamental group $G_a$.  Using the
definition of multiplication, 
the first component of
$(\mu_{a,H},R_H)^{n_H}$ is the path got by following $\mu_{a,H}$, then
$R_H(\mu_{a,H})$, then $R_H^2(\mu_{a,H}),\ldots$ and finally
$R_H^{n_H-1}(\mu_{a,H})$.  
It is easy to see that this is homotopic to $\Loop{a H}$.
So we have $(\mu_{a,H},R_H)^{n_H}=\Loop{a H}$.

\medskip
At this point we have defined everything in the statement of
theorem~\ref{thm-not-quite-strong-enough}.  But before proving it, we will adapt our construction to
accomodate the ``fat basepoints'' of the previous section.  This is
necessary for our application to $\ch^{13}$.  
So we fix $A$ as in
section~\ref{sec-loops-in-arrangement-complements}, and assume it contains our basepoint
$a$.  We will use $A$ as the base ``point'' when discussing
$\pi_1(X-\H)$, and $a$ as the basepoint when discussing
$\piorb\big((X-\H)/\PG\bigr)$.  In particular, the left term of
\eqref{eq-exact-sequence-on-pi-1} could also be written $\pi_1(X-\H,A)$.  The analogue of
$\mu_{a,H}$ is defined as follows, in terms of the point $b$ of $A$
that is closest to $H$.  We define $\mu_{a,A,H}$ to be $\geodesic{ab}$
followed by $\mu_{b,H}$ followed by $R_H(\geodesic{b a})$.  See
figure~\ref{fig-illustration-of-mu} for a picture.  Essentially the same argument as before
shows that $(\mu_{a,A,H},R_H)^{n_H}=\Loop{AH}\in\pi_1(X-\H,A)$.

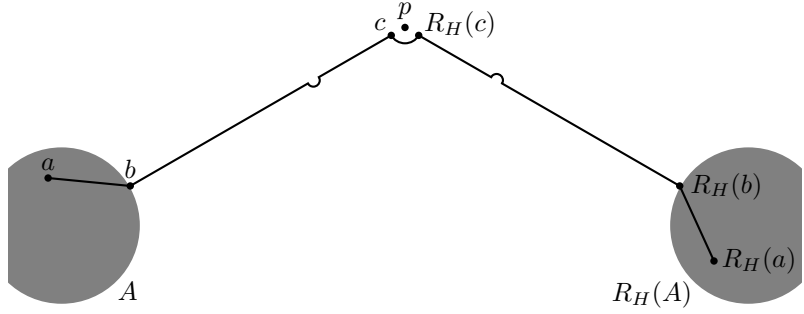
\begin{figure}
\psset{unit=30pt}
\def\URx{5}
\def\URy{.5}
\def\LLx{-\URx}
\def\LLy{-3.7}
\begin{pspicture*}(\LLx,\LLy)(\URx,\URy)
\def\px{0}
\def\py{0}
\def\pradius{.2}
\def\dotradius{.06}
\def\RAx{4.33}
\def\RAy{-2.5}
\def\Aradius{1}
\def\ax{-4.5}
\def\ay{-1.9}
\def\Rdetourx{1.155}
\def\Rdetoury{-.667}
\def\detourradius{.08}
\def\Rbx{3.464}
\def\Rby{-2}
\def\Rcx{.173}
\def\Rcy{-.1}
\def\Rax{3.896}
\def\Ray{-2.947}
\psline( \RAx, \RAy)( \px, \py)(-\RAx, \RAy)
\pscircle[linestyle=none,fillstyle=solid,fillcolor=gray](\RAx,\RAy){\Aradius}
\pscircle[linestyle=none,fillstyle=solid,fillcolor=gray](-\RAx,\RAy){\Aradius}
\pscircle[linestyle=none,fillstyle=solid,fillcolor=white](\px,\py){\pradius}
\pscircle[linestyle=none,fillstyle=solid,fillcolor=white](\Rdetourx,\Rdetoury){\detourradius}
\pscircle[linestyle=none,fillstyle=solid,fillcolor=white](-\Rdetourx,\Rdetoury){\detourradius}
\psarc[linestyle=solid,fillstyle=none]( \Rdetourx,\Rdetoury){\detourradius}{330}{150}
\psarc[linestyle=solid,fillstyle=none](-\Rdetourx,\Rdetoury){\detourradius}{210}{30}
\psarc[linestyle=solid,fillstyle=none](\px,\py){\pradius}{210}{330}
\psline( \Rax, \Ray)( \Rbx, \Rby)
\psline( \ax, \ay)(-\Rbx, \Rby)
\pscircle[linestyle=none,fillstyle=solid,fillcolor=black]( \ax ,  \ay){\dotradius}
\uput{3pt}[90](  \ax,  \ay){$a$}
\pscircle[linestyle=none,fillstyle=solid,fillcolor=black](-\Rbx, \Rby){\dotradius}
\uput{3pt}[90](-\Rbx, \Rby){$b$}
\pscircle[linestyle=none,fillstyle=solid,fillcolor=black](-\Rcx, \Rcy){\dotradius}
\uput{3pt}[135](-\Rcx, \Rcy){$c$}
\pscircle[linestyle=none,fillstyle=solid,fillcolor=black]( \px, \py){\dotradius}
\uput{3pt}[90]( \px, \py){$p$}
\pscircle[linestyle=none,fillstyle=solid,fillcolor=black]( \Rcx, \Rcy){\dotradius}
\uput{2pt}[30]( \Rcx, \Rcy){$R_H(c)$}
\pscircle[linestyle=none,fillstyle=solid,fillcolor=black]( \Rbx, \Rby){\dotradius}
\uput{4pt}[0]( \Rbx, \Rby){$R_H(b)$}
\pscircle[linestyle=none,fillstyle=solid,fillcolor=black]( \Rax, \Ray){\dotradius}
\uput{3pt}[0]( \Rax, \Ray){$R_H(a)$}
\uput{30pt}[-45](-\RAx,\RAy){$A$}
\uput{30pt}[-135](\RAx,\RAy){$R_H(A)$}
\end{pspicture*}
\caption{The path $\mu_{a,A,H}$ goes from left to right.  Here 
  $R_H$ is the complex reflection of order~$3$, acting by counter-clockwise
  rotation by $2\pi/3$.   The hyperplane $H$ is not shown
  except for its point~$p$ closest to~$A$. The small semicircles
  indicate that the path from $b$ to $c$ may detour around
  some points of $\H$.}
\label{fig-illustration-of-mu}
\end{figure}

In applications one typically has some distinguished set of
$\mu_{a,H}$'s or $\mu_{a,A,H}$'s in mind and wants to prove that they
generate $G_a$.  Theorem~\ref{thm-not-quite-strong-enough} in the
introduction is a result of this sort, and the rest of the section is
devoted to proving it.  The following lemma is really
the inductive step in the proof, so the reader might prefer to read
the theorem's proof first.  Also, theorem~\ref{thm-not-quite-strong-enough} uses only
case~\eqref{item-can-move-A-closer-to-p} of the lemma; the other cases
are for our application to $\ch^{13}$ in section~\ref{sec-Leech-meridians-generate}.

\begin{lemma}
\label{lem-lowering-lemma}
Suppose $\calC\sset\M$ are the hyperplanes closest to~$A$,
and let $G$ be the subgroup of
$G_a=\piorb\bigl((X-\H)/\PG,a\bigr)$ generated by the $(\mu_{a,A,C},R_C)$ with
$C\in\calC$.  
Suppose $H\in\M$, write $p$
for the closest point of $H$ to $A$, $r$ for $d(A,p)$, and $B$ for the
open $r$-neighborhood of $A$.  Suppose  $G$ contains
$\pi_1(B-\H,A)$ and that there exists a complex reflection $R\in\PG$ with
mirror in $\calC$, such 
that one of the following holds:
\begin{enumerate}
\item
\label{item-can-move-A-closer-to-p}
$R$ moves $A$ closer
to $p$.
\item
\label{item-can-move-A-closer-to-H-and-no-further-from-A}
$R$ moves $A$ closer
to $H$, and no farther from $p$.
\item
\label{item-more-complicated-move-closer-condition}
There exists
an open ball $U$ around $p$ such that
$U\cap\H=U\cap\H_p$, $B\cap R(B)\cap U\neq\emptyset$, and $R(B)\cap
U\cap H\neq\emptyset$.
\end{enumerate}
Then $G$ contains $\Loop{AH}$.
\end{lemma}

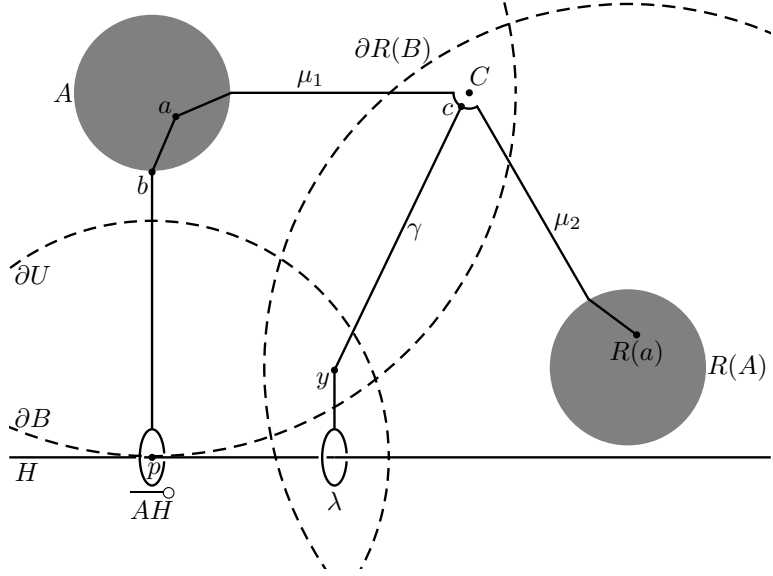
\begin{figure}
\psset{unit=30pt}
\def\URx{3.8}
\def\URy{1.2}
\def\LLx{-5.8}
\def\LLy{-6}
\begin{pspicture*}(\LLx,\LLy)(\URx,\URy)
\def\thinwidth{.03}
\psset{linewidth=\thinwidth}
\def\whiteellipseheight{.4}
\def\whiteellipsewidth{.2}
\def\blackellipseheight{.37}
\def\blackellipsewidth{.17}
\def\whitegap{.1}
\def\dotradius{.06}
\def\Lx{0}
\def\Ly{0}
\def\Lradius{.2}
\def\Ax{-4}
\def\Ay{0}
\def\Aradius{1}
\def\RAx{2}
\def\RAy{-3.464}
\def\Bradius{4.6}
\def\px{\Ax}
\def\py{-\Bradius}
\def\Uradius{3}
\def\ax{-3.7}
\def\ay{-.3}
\def\bx{\Ax}
\def\by{-\Aradius}
\def\yx{-1.7}
\def\yy{-3.5}
\def\abovepy{-4.23} 
\def\cx{-.1}\def\cy{-.173} 
\def\bx{\px}\def\by{-\Aradius} 
\def\gammax{-.85}\def\gammay{-1.75} 
\def\rightofAx{-3} 
\def\Rax{2.11}\def\Ray{-3.054} 
\def\topleftRAx{1.5}\def\topleftRAy{-2.6} 
\def\muonex{-2}\def\muoney{0} 
\def\mutwox{1}\def\mutwoy{-1.732} 
\psline(\LLx,\py)(\URx,\py)
\psellipse[linestyle=solid,linewidth=\whitegap,fillstyle=none,linecolor=white](\px,\py)(\whiteellipsewidth,\whiteellipseheight)
\psellipse[linestyle=solid,fillstyle=none,linecolor=black](\px,\py)(\blackellipsewidth,\blackellipseheight)
\psline[linewidth=\whitegap,linecolor=white](\px,\py)(\URx,\py)
\psline(\px,\py)(\URx,\py)
\psellipse[linestyle=solid,linewidth=\whitegap,fillstyle=none,linecolor=white](\yx,\py)(\whiteellipsewidth,\whiteellipseheight)
\psellipse[linestyle=solid,fillstyle=none,linecolor=black](\yx,\py)(\blackellipsewidth,\blackellipseheight)
\psline[linewidth=\whitegap,linecolor=white](\yx,\py)(\URx,\py)
\psline(\yx,\py)(\URx,\py)
\psline(\Ax,\Ay)(\px,\abovepy)
\psline(\yx,\yy)(\yx,\abovepy)
\pscircle[linestyle=none,fillstyle=solid,fillcolor=black](\px,\py){\dotradius}
\uput{2pt}[-80](  \px,  \py){$p$}
\pscircle[linestyle=none,fillstyle=solid,fillcolor=black](\yx,\yy){\dotradius}
\uput{2pt}[-140](  \yx,  \yy){$y$}
\uput{3pt}[-45](\LLx,\py){$H$}
\uput{12pt}[-90](\px,\py){$\Loop{AH}$}
\uput{13pt}[-90](\yx,\py){$\lambda$}
\psline( \Ax, \Ay)( \Lx, \Ly)(\RAx, \RAy)
\pscircle[linestyle=none,fillstyle=solid,fillcolor=gray](\Ax,\Ay){\Aradius}
\pscircle[linestyle=none,fillstyle=solid,fillcolor=gray](\RAx,\RAy){\Aradius}
\pscircle[linestyle=none,fillstyle=solid,fillcolor=white](\Lx,\Ly){\Lradius}
\psline(\ax,\ay)(\rightofAx,\Ay)
\psline(\topleftRAx,\topleftRAy)(\Rax,\Ray)
\pscircle[linestyle=none,fillstyle=solid,fillcolor=black](\ax,\ay){\dotradius}
\uput{2pt}[135](  \ax,  \ay){$a$}
\pscircle[linestyle=none,fillstyle=solid,fillcolor=black](\Rax,\Ray){\dotradius}
\uput{2pt}[-90](  \Rax,  \Ray){$R(a)$}
\psarc[linestyle=solid,fillstyle=none](\Lx,\Ly){\Lradius}{180}{300}
\pscircle[linestyle=none,fillstyle=solid,fillcolor=black](\Lx,\Ly){\dotradius}
\uput{3pt}[60](  \Lx,  \Ly){$C$}
\pscircle[linestyle=none,fillstyle=solid,fillcolor=black](\cx,\cy){\dotradius}
\uput{3pt}[190](  \cx,  \cy){$c$}
\uput{2pt}[90](\muonex,\muoney){$\mu_1$}
\uput{2pt}[30](\mutwox,\mutwoy){$\mu_2$}
\uput{30pt}[180](\Ax,\Ay){$A$}
\uput{30pt}[0](\RAx,\RAy){$R(A)$}
\psline(\ax,\ay)(\bx,\by)
\pscircle[linestyle=none,fillstyle=solid,fillcolor=black](\bx,\by){\dotradius}
\uput{2pt}[-135](  \bx,  \by){$b$}
\psline(\cx,\cy)(\yx,\yy)
\uput{2pt}[0](\gammax,\gammay){$\gamma$}
\pscircle[linestyle=dashed,fillstyle=none](\Ax,\Ay){\Bradius}
\uput{2pt}[0](\LLx,-4.1){$\partial B$}
\pscircle[linestyle=dashed,fillstyle=none](\RAx,\RAy){\Bradius}
\uput{2pt}[150](-.4,.4){$\partial R(B)$}
\pscircle[linestyle=dashed,fillstyle=none](\px,\py){\Uradius}
\uput{2pt}[0](\LLx,-2.3){$\partial U$}
\end{pspicture*}
\caption{Illustration for the proof of lemma~\ref{lem-lowering-lemma}.}
\label{fig-the-essential-homotopy}
\end{figure}

\begin{proof}
In every case we have $R(B)\cap H\neq\emptyset$, so $R^{-1}(H)$ is
closer to $A$ than $H$ is.  Therefore
$H\notin\calC$, or in other words:
 the hyperplanes in $\calC$ lie at
distance${}<r$ from $A$.
We will prove the
lemma under hypothesis \eqref{item-more-complicated-move-closer-condition}, and then show that the other two
cases follow.  We hope figure~\ref{fig-the-essential-homotopy} helps
the reader.  First we introduce the various objects pictured.  As
we did above, we
write $b$ for the point of $A$ closest to $H$.  Under our
identification of $\pi_1(X-\H,A)$ with $\pi_1(X-\H,a)$, the loop
$\Loop{AH}$ corresponds to $\geodesic{ab}$ followed by $\Loop{b H}$
followed by $\geodesic{b a}$.

The complex reflection $R$ equals $R_C^i$ for some $C\in\calC$.  The
point marked $C$ in the figure represents the point of $C$ nearest to
$A$.  It lies inside $B$ by the previous paragraph's remark that the elements of $\calC$
are closer to $A$ than $H$ is.  It also lies in $R(B)$, since $R$
fixes $C$ pointwise.

Consider the first component of $(\mu_{a,A,C},R_C)^i$.  
After a homotopy
it may be regarded as a path $\mu_1$ in $B-\H$ from $a$ to a
point $c\in \bigl(B\cap R(B)\bigr)-\H$, followed by a path $\mu_2$ in $R(B)-\H$ from
$c$ to $R(a)$.  These paths are marked in the figure.
So $(\mu_{a,A,C},R_C)^i=(\mu_1\mu_2,R)$ in $G_a$. 

The hypothesis that $B\cap R(B)\cap U\neq\emptyset$ is exactly what we
need to know that some point $y$ lies in this intersection, as drawn.
The connectedness of $U\cap R(B)$ and the hypothesis that $H$ meets
$U\cap R(B)$ are exactly what we need to construct a loop $\lambda$ in
$(R(B)\cap U)-\H$, based at $y$, with linking number~$1$ with $H$.
Finally, the connectedness of $B\cap R(B)$ allows us to construct a
path $\gamma$ in $(B\cap R(B))-\H$ from $c$ to~$y$.  This finishes the
construction of the objects in the figure.

Our goal is to prove that $G$ contains $\Loop{AH}$.
Lemma~\ref{lem-generators-for-pi-1-of-ball-plus-small-bulge} shows
that this loop lies in the subgroup of $\pi_1(X-\H,a)$ generated by
$\pi_1(B-\H,a)$ and $\mu_1\gamma\lambda\gamma^{-1}\mu_1^{-1}$.  This
uses our hypothesis $U\cap\H=U\cap\H_p$.  Since we assumed $G$
contains the image of $\pi_1(B-\H,a)$, it suffices to show that $G$
contains $\mu_1\gamma\lambda\gamma^{-1}\mu_1^{-1}$, or equivalently
the homotopic
loop
$(\mu_1\mu_2)\bigl(\mu_2^{-1}\gamma\lambda\gamma^{-1}\mu_2\bigr)(\mu_2^{-1}\mu_1^{-1})$.

An element of the orbifold fundamental group $G_a$ is really a pair,
so we must prove
$\bigl((\mu_1\mu_2)(\mu_2^{-1}\gamma\lambda\gamma^{-1}\mu_2)(\mu_2^{-1}\mu_1^{-1}),1\bigr)\in G$.  One
checks that this
equals
$$(\mu_1\mu_2,R)\cdot\Bigl(R^{-1}\bigl(\mu_2^{-1}\gamma\lambda\gamma^{-1}\mu_2\bigr),1\Bigr)\cdot\Bigl(R^{-1}\bigl(\mu_2^{-1}\mu_1^{-1}\bigr),R^{-1}\Bigr)$$
The last term is the inverse of the first, which $G$ contains by
definition.
So it suffices
to show that the middle term lies in $G$, which is easy: 
the loop $\mu_2^{-1}\gamma\lambda\gamma^{-1}\mu_2$ lies in $R(B)-\H$, so its image
under $R^{-1}$ lies in $B-\H$.  This finishes case \eqref{item-more-complicated-move-closer-condition}.

Next we claim that \eqref{item-can-move-A-closer-to-p} implies \eqref{item-more-complicated-move-closer-condition}.  Take $U$ to be any ball
around $p$ with $U\cap\H=U\cap\H_p$.  
Then the remaining hypotheses of
\eqref{item-more-complicated-move-closer-condition} follow immediately from $p\in R(B)$.  

Finally we claim that \eqref{item-can-move-A-closer-to-H-and-no-further-from-A} implies \eqref{item-more-complicated-move-closer-condition}.  By the previous
paragraph it suffices to treat the case that $p\in\partial R(B)$.
Take $U$ to be any ball around $p$ with $U\cap\H=U\cap\H_p$.  The
hypothesis $d(R(A),H)<r$ says that $H$ is not orthogonal to
$\geodesic{R(A),p}$.  
It follows that $R(B)$ contains elements of $H$
arbitrarily close to $p$, so $U\cap R(B)\cap H\neq\emptyset$.    
Similarly, $d(R(A),H)<r$ implies the non-tangency of
$\partial B$ and $\partial
R(B)$ at $p$.  From this it follows that
$B\cap R(B)$ has elements arbitrarily close to $p$, hence in $U$.
This finishes the proof.
\end{proof}

\begin{proof}[Proof of theorem~\ref{thm-not-quite-strong-enough}]
We will apply  lemma~\ref{lem-lowering-lemma} with $A=\{a\}$, noting that $\mu_{a,A,H}=\mu_{a,H}$ for
all $H\in\M$. 
Write $G$ for the subgroup of $G_a$ generated by the $(\mu_{a,C},R_C)$'s.
By the exact sequence \eqref{eq-exact-sequence-on-pi-1} and the assumed surjectivity $G\to\PG$,
it suffices to show that $G$ contains $\pi_1(X-\H,a)$.  By
theorem~\ref{thm-generators-for-X-minus-hyperplanes} it suffices to show that it contains every $\Loop{a H}$.
We do this by induction on $d(a,H)$.  

The base case is $H\in\calC$,
for which we use the fact that $\Loop{a H}$ is a power of $(\mu_{a,H},R_H)$.
So suppose $H\in\M-\calC$ and set $r:=d(a,H)$.  We may assume, by
theorem~\ref{thm-generators-for-metric-neighborhood-of-A} and the inductive hypothesis, that $G$ contains
$\pi_1(B-\H,a)$, where $B$ is the open $r$-neighborhood of $a$.  Then
case \eqref{item-can-move-A-closer-to-p} of lemma~\ref{lem-lowering-lemma}  shows that $G$ also contains
$\Loop{a H}$, completing the inductive step.
\end{proof}

\section{A monstrous(?) hyperplane arrangement}
\label{sec-background}

\noindent
In this section we give background information on the
conjecturally-monstrous hyperplane arrangement in $\ch^{13}$ which is
the subject of conjecture~\ref{conj-monstrous-proposal} and theorems \ref{thm-announcement} and~\ref{thm-leech-generation}.  For more information, see
\cite{Allcock-monstrous}\cite{Basak-bimonster-1}\cite{Basak-bimonster-2}\cite{Allcock-y555}\cite{Heckman}\cite{Heckman-Rieken}.

We write $\C^{n,1}$ for a complex vector space equipped with a
Hermitian form $\ip{\cdot}{\cdot}$ of signature $(n,1)$,  
assumed linear in its first argument and antilinear in its second.  The {\it
  norm} $v^2$ of a vector $v$ means $\ip{v}{v}$.  Complex
hyperbolic space $\chn$ means the set of negative-definite
1-dimensional subspaces.  If $V,W\in\ch^n$ are represented by
vectors $v,w$ then their hyperbolic distance is
\begin{equation}
\label{eq-distance-formula-for-points}
d(V,W)=\cosh^{-1}\sqrt{\frac{\bigl|\ip{v}{w}\bigr|^2}{v^2 w^2}}
\end{equation}
If $s$ is a vector of positive norm, then $s^\perp\sset\C^{n,1}$
defines a hyperplane in $\chn$, also written $s^\perp$, and
\begin{equation}
\label{eq-distance-formula-for-point-and-mirror}
d(V,s^\perp)=\sinh^{-1}\sqrt{-\frac{\bigl|\ip{v}{s}\bigr|^2}{v^2 s^2}}
\end{equation}
These formulas are from \cite{Goldman}, 
up to an unimportant factor of~$2$.

A null vector means a nonzero vector of norm~$0$.
If $v$ is one then it represents a point $V$ of the
boundary $\partial\chn$.
For any vector $w$ of non-zero norm  we define the {\it height} of $w$ with
respect to $v$ by
\begin{equation}
\label{eq-definition-of-height}
\height_v(w):=-\frac{\bigl|\ip{v}{w}\bigr|^2}{w^2}.
\end{equation}
This function is invariant under rescaling $w$, so it descends to a
function on $\chn$, which is positive.  The horosphere
centered at $V$, of height~$h$ 
with respect to $v$, means the set of $p\in\chn$ with
$\height_v(p)=h$.  We define open and closed horoballs the same way,
replacing $=$ by $<$ and $\leq$.  (More abstractly, one can define
horospheres as the orbits of the unipotent radical of the
$\PU(n,1)$-stabilizer of $V$.)

We think of $V$ as the center of these horospheres and horoballs and
$h$ as a sort of generalized radius, even though strictly speaking the
distance from any point of $\chn$ to $V$ is infinite.  In particular,
if $p,p'\in\chn$ then we say that $p$ is closer to $V$ than $p'$ is,
if $\height_v(p)<\height_v(p')$.  To see that this notion depends on
$V$ rather than $v$, one checks that replacing $v$ by a nonzero scalar
multiple of itself does not affect this inequality.  (It multiplies
both sides by the same positive number.)  Another way to think about
this, at least for points outside some fixed closed horoball~$A$
centered at $V$, is to
regard ``closer to~$V$'' as alternate language for ``closer to~$A$''.  In
any case, in our application there will be a canonical choice for $v$,
up to roots of unity.

\medskip Next we will describe the hyperplane arrangement appearing in
conjecture~\ref{conj-monstrous-proposal} and theorems~\ref{thm-announcement}--\ref{thm-leech-generation}.  We write $\w$ for a
primitive cube root of unity and define the Eisenstein integers $\E$
as $\Z[\w]$.  The Eisenstein integer $\w-\wbar=\sqrt{-3}$ is so
important that it has its own name~$\theta$.  An {\it $\E$-lattice}
means a free $\E$-module $L$ equipped with a Hermitian form taking
values in $\E\tensor\Q=\Q(\sqrt{-3})$, denoted $\ip{\cdot}{\cdot}$.
Sometimes we think of lattice elements as column vectors and
$\ip{\cdot}{\cdot}$ as specified by a matrix $M$ equal to the
transpose of its complex conjugate.  Then $\ip{v}{w}=v^T M\bar{w}$.

We will describe two $\E$-lattices, from \cite{Allcock-y555} and \cite{Allcock-Inventiones}.  Each
has signature $(13,1)$ and is equal to $\theta$ times its dual
lattice.  By \cite{Basak-bimonster-1} there is only one lattice with
these properties, so we may regard them as two different descriptions
of the same lattice~$L$.  We will not actually use this uniqueness in
this paper, and the first description of~$L$ is presented only to make
precise the statement of theorem~\ref{thm-announcement}.

The definitions of $\PG$ and $\M$, and some language we will use, are
independent of the model.  We regard $L\tensor_\E\C$ as a
copy of $\C^{13,1}$, and take $\G$ to be the group of $\E$-linear
automorphisms of $L$ that preserve the inner product.  As usual, $\PG$
means the quotient by its subgroup of scalars.  A {\it root} means a
norm~$3$ lattice vector, the hyperplane arrangement $\M$ consists of
the orthogonal complements in $\ch^{13}$ of the roots, and $\H$ means
the union of these hyperplanes.  The subject of conjecture~\ref{conj-monstrous-proposal} is the
orbifold fundamental group of $(\ch^{13}-\H)/\PG$.

The special role of norm~$3$ vectors, and the name ``root'', arises
as follows.  First let $s\in\C^{13,1}$ be any vector of positive norm.
Then the linear map
\begin{equation*}
x\mapsto x-(1-\w)\frac{\ip{x}{s}}{s^2}s
\end{equation*}
is an isometry of $\ip{}{}$, called the {\it $\w$-reflection in $s$}
and denoted $R_s$.  Replacing $\w$ by $\wbar$ gives the
$\wbar$-reflection in $s$, which is the inverse of $R_s$.  They are
called triflections, because they are complex reflections of
order~$3$.  To see that $R_s$ is a complex reflection (in particular
an isometry) one checks that it fixes $s^\perp$ pointwise and
multiplies $s$ by $\w$.  

In the special case that $s$ is a root, $R_s$
preserves~$L$ because of a conspiracy among the coefficients.  First,
the factor $(1-\w)$ is a unit multiple of $\theta=\sqrt{-3}$.  Second,
for any lattice vector $x$, $\ip{x}{s}$ is divisible by $\theta$,
since all inner products in $L$ are.  (This is what it means for $L$
to lie in $\theta$ times its dual lattice.)  Together these two
factors of $\theta$ cancel the $s^2=3$ term in the denominator, up to
a unit.  So $R_s(x)$ is an $\E$-linear combination of $x$ and $s$,
hence lies in $L$.  When one has a reflection in mind (real or
complex), it is customary to call a vector orthogonal to its
fixed-point set a root.  When one also has a lattice in mind, one usually
fixes the scale of a root by requiring it to be a primitive lattice
vector.  This is why we call norm~$3$ vectors roots.  One can show
that no other elements of $\PG$ act on $\ch^{13}$ by complex
reflections.  (The analogous result for unimodular $\E$-lattices is
contained in lemmas~8.1--8.2 of \cite{Allcock-New}; for the current case one
uses the fact that $L$ is equal to
$\theta$ times its dual lattice, rather than merely lying in it.)  So
$\M$ is exactly the set of mirrors of the complex reflections in
$\PG$, making $\piorb\bigl((\ch^{13}-\H)/\PG\bigr)$ a braid-like group
in the sense of this paper.

The second author showed in \cite[Lemma~3.3]{Basak-bimonster-2} that the
triflections in a particular set of~$14$ roots generate $\G$.  The first author showed in \cite{Allcock-Inventiones} that all roots are
equivalent under $\G$.

\medskip
Our first description of $L$ is the ``$P^2\F_3$ model'' from
\cite{Allcock-y555}.  As mentioned above, we include it only to give
precise meaning to theorem~\ref{thm-announcement}, and we will not
refer to it later. We start with the diagonal inner product matrix
$[-1;1,\dots,1]$ on $\C^{13,1}$, and regard the last $13$ coordinates
as being indexed by the $13$ points of $P^2\F_3$.  The ``point roots''
are the vectors of the form $(0;\theta,0^{12})$ with the $\theta$ in
any of last~$13$ positions.  The ``line roots'' are the vectors of the
form $(1;1,1,1,1,0,\dots,0)$, with $1$'s in positions corresponding to
the points of a line in $P^2\F_3$.  $L$ is defined as the span of
these~$26$ roots.  This construction obviously has $\PGL_3(\F_3)$
symmetry, and it also has less-obvious symmetries exchanging the point
roots and line roots (up to scalars).  This yields a subgroup
$\PGL_3(\F_3)\rtimes\Z/2$ of $\PG$ that acts transitively on these
$26$ roots.  We take $\tau$ in theorem~\ref{thm-announcement} to be
the unique point of $\ch^{13}$ invariant under this group; its
coordinates are $(4+\sqrt3;1,\dots,1)$.  Among other results, it was
shown in \cite{Basak-bimonster-1}, Prop.\ 1.2,  that the hyperplanes in
$\M$ that are closest to $\tau$ are exactly the mirrors of the point
and line roots.  We have now described concretely all the objects in
theorem~\ref{thm-announcement}.

\medskip
Now we give the ``Leech model'' of $L$ from \cite{Allcock-Inventiones}, and make concrete the
objects in theorem~\ref{thm-leech-generation}.  We will use this model for the rest of the
paper.  We define $L$ as the $\E$-lattice $\Lambda\oplus\cell$, where
$\Leech$ is the complex Leech lattice at the smallest scale at which
all inner products lie in~$\E$.  The complex Leech lattice is studied
in detail in \cite{Wilson}; the scale used there is the most natural
one for coordinate computations in it, and has minimal norm~$9$.  At
our scale it has minimal norm~$6$, all inner products are divisible
by~$\theta$, and $\Lambda$ is equal to $\theta$ times its dual
lattice.  It is called the complex Leech lattice because its
underlying real lattice is a scaled copy of the usual (real) Leech
lattice described in \cite{Conway}.  The properties of $\Lambda$ that we
will use are that its automorphism group is transitive on its vectors
of norms~$6$ and~$9$, and that its covering radius is~$\sqrt3$.  The
transitivity is proven in \cite{Wilson}.  The meaning of the
covering radius is that closed balls of that radius, centered at
lattice points, exactly cover Euclidean space.  It is $\sqrt3$ because
the real Leech lattice, at its own natural scale, has minimal norm~$4$ and
(by \cite{Conway-Parker-Sloane}) covering radius~$\sqrt2$.

We will write vectors of $L$ in the form $(x;y,z)$, where
$x\in\Lambda$ and $y,z\in\E$.

Conceptually, our ``basepoint'' for the description of generators for
$\piorb\bigl((\ch^{13}-\H)/\PG\bigr)$ in theorem~\ref{thm-leech-generation} is the cusp of
$\PG$ represented by the null vector $\rho:=(0;0,1)$.  As explained in
the introduction, really this is a shorthand for choosing a ``fat
basepoint'': a closed horoball $A$ centered at $\rho$ that misses
$\H$.  The following lemma shows that such a horoball exists.  More
precisely, it identifies the largest open horoball centered at $\rho$
that misses $\H$; we may take $A$ to be any closed horoball inside it.
We also fix a basepoint $a\in A$, so now the paths $\mu_{a,A,H}$ in
theorem~\ref{thm-leech-generation} have been defined.

\begin{lemma}
\label{lem-disjoint-horoball-exists}
The open horoball
$\bigl\{p\in\ch^{13}\bigm|\height_\rho(p)<1\bigr\}$ is disjoint from
$\H$, and the mirrors that meet its boundary are the orthogonal
complements of the roots $l$ that
satisfy $\bigl|\ip{\rho}{l}\bigr|^2=3$.
\end{lemma}

\begin{proof}
The special property of $\rho$ we need is that it is orthogonal to no
roots.  This is clear because $\rho^\perp=\Leech\oplus\langle\rho\rangle$ has no
vectors of norm~$3$.  Now, if $l$ is a root then the point of $l$'s
mirror nearest to $\rho$ is represented by the vector projection of
$\rho$ to $l^\perp$, namely $p=\rho-\frac{1}{3}\ip{\rho}{l}l$.  One
computes
$\height_\rho(p)=\bigl|\ip{\rho}{l}\bigr|^2/3$.  This is at least~$1$, with equality just
if $\bigl|\ip{\rho}{l}\bigr|^2=3$.
\end{proof}

Lemma~\ref{lem-disjoint-horoball-exists} also identifies the mirrors closest to $\rho$, namely the
orthogonal complements of the roots $l$ with $\ip{\rho}{l}$ equal to a
unit multiple of $\theta$.  After scaling $l$ we may suppose
\begin{equation}
\label{eq-definition-of-a-leech-root-l}
l
=
\biggl(
\lambda
;
1,
\theta\Bigl(\frac{\lambda^2-3}{6}+\nu_l\Bigr)
\biggr)
\end{equation}
where $\lambda\in\Leech$ and $\nu_l$ is purely imaginary and chosen so
that the last coordinate lies in $\E$.  The set of possibilities for
$\nu_l$ is $\frac{1}{\theta}(\frac{1}{2}+\Z)$ if $6$ divides
$\lambda^2$ and $\frac{1}{\theta}\Z$ otherwise.  Despite its elaborate form,
the last coordinate is well-suited for the calculations required
in next section.  We call these roots the {\it Leech roots} (hence
the notation $l$), their mirrors the {\it Leech mirrors} and the
meridians $(\mu_{a,A,l^\perp},R_l)$ the {\it Leech meridians}.
We have now made concrete all the objects in theorem~\ref{thm-leech-generation}.

We remark that there is a $25$-dimensional integral Heisenberg group
in $\PG$ that acts simply-transitively on the Leech roots.  It
consists of the ``translations'' in the proof of lemma~\ref{lem-classification-of-roots-of-height-theta}.
Conceptually, this is a simpler way to index the Leech roots than by
the pairs $\lambda,\nu_l$, but in the end it is equivalent.  Also,
these translations act cocompactly on $\partial A$, verifying the
technical condition we required on $A$ in section~\ref{sec-loops-in-arrangement-complements}.

\smallskip
Early in the section we explained how one can meaningfully say that
one point of $\ch^{13}$ is closer to a point of $\partial\ch^{13}$
than another point of $\ch^{13}$ is.  We will also need to be able to
compare the ``distance'' from a point $p\in\ch^{13}$ to two different
cusps $V,V'\in\partial\ch^{13}$ of $\PG$.  Being cusps, they can be
represented by lattice vectors $v$, $v'$, which we may choose to be
primitive.  Then $v,v'$ are well-defined up to multiplication by sixth
roots of unity, and the corresponding height functions $\height_v$,
$\height_{v'}$ are independent of these factors.  So we will say that
$p$ is closer to $V$ than to $V'$ if $\height_v(p)<\height_{v'}(p)$.
Note that this construction depends on the fact that $V,V'$ are cusps;
it does not make sense for general points of $\partial\ch^{13}$.  In
our applications, $v$ and $v'$ will always be $\rho$ and a
translate of~$\rho$.

\section{The Leech meridians generate}
\label{sec-Leech-meridians-generate}

\noindent
The purpose of this section is to prove
theorem~\ref{thm-leech-generation}, showing that the Leech meridians
generate the orbifold fundamental group
$G_a:=\piorb\bigl((\ch^{13}-\H)/\PG,a\bigr)$, where $\M$, $\H$, $\PG$,
$A$, $a$ and the Leech meridians $(\mu_{a,A,l^\perp},R_l)$ are defined
in the previous section.  We write $G$ for the subgroup of $G_a$ they
generate, and we must prove that
$G$ is all of~$G_a$.

We begin with an overview of the proof, which follows that of
theorem~\ref{thm-not-quite-strong-enough}.  It amounts to showing that
the mirror of any non-Leech root $s$ satisfies one of the hypotheses
\eqref{item-can-move-A-closer-to-p}--\eqref{item-more-complicated-move-closer-condition}
of lemma~\ref{lem-lowering-lemma}.  It turns out
(lemma~\ref{lem-some-Leech-reflection-moves-cusp-closer-to-projection-point})
that if $\bigl|\ip{\rho}{s}\bigr|^2>21$ then the simplest hypothesis
\eqref{item-can-move-A-closer-to-p} holds.  If
$\bigl|\ip{\rho}{s}\bigr|^2>9$ then the same method shows that the
next simplest hypothesis
\eqref{item-can-move-A-closer-to-H-and-no-further-from-A} holds
(lemmas
\ref{lem-some-Leech-reflection-moves-cusp-closer-to-projection-point}
and~\ref{lem-height-reduction-almost-moves-rho-closer-to-projection-point}).
For the case $\bigl|\ip{\rho}{s}\bigr|^2=9$ we enumerate the orbits of
roots $(?;\theta,?)$ under the $\G$-stabilizer of~$\rho$
(lemma~\ref{lem-classification-of-roots-of-height-theta}).  There are
three orbits, satisfying hypotheses
\eqref{item-can-move-A-closer-to-p},
\eqref{item-can-move-A-closer-to-H-and-no-further-from-A} and
\eqref{item-more-complicated-move-closer-condition} of
lemma~\ref{lem-lowering-lemma}, respectively.  The last orbit is
especially troublesome
(lemma~\ref{lem-hole-meridians-generated-by-leech-meridians}).  The
proof of theorem~\ref{thm-leech-generation} is then a wrapper
around these results.

The following description of vectors in $L\tensor\C$  is very
important in our computations.   We will use it
constantly, often specializing to the case of roots.  Every vector $s\in
(L\tensor\C)-\rho^\perp$ can be written uniquely in the form
\begin{equation}
\label{eq-definition-of-vector-s}
s
=
\biggl(
\sigma
;
m
,
\frac{\theta}{\mbar}\Bigl(\frac{\sigma^2-N}{6}+\nu\Bigr)
\biggr)
\end{equation}
where $\sigma\in\Leech\tensor\C$, $m\in\C-\{0\}$, $N$ is the norm
$s^2$, and $\nu$ is purely imaginary.  Restricting the first coordinate to
$\Leech$ and the others to $\E$ gives the elements of $L-\rho^\perp$.
Further restricting $N$ to~$3$ gives the roots of $L$, and finally
restricting $m$ to~$1$ gives the Leech roots from
\eqref{eq-definition-of-a-leech-root-l}.  For vectors of any fixed negative (resp.\ positive) norm,
the larger the absolute value of the middle coordinate~$m$, the
further from $\rho$ lie the corresponding points (resp.\ hyperplanes)
in $\ch^{13}$.

One should think of $s$ from \eqref{eq-definition-of-vector-s} as being  associated to the
vector $\sigma/m$ in the positive-definite Hermitian vector space $\Lambda\tensor_\E\C$.  By this we mean that the most important
part of $\ip{s}{s'}$ is governed by the
relative positions of $\sigma/m$ and $\sigma'/m'$.  Namely, by writing
out $\ip{s}{s'}$, completing the square and
patiently rearranging, one can check
\begin{equation}
\label{eq-general-inner-product-formula}
\ip{s}{s'}
=
m\mbar'
\biggl[
\frac{1}{2}
\biggl(
\frac{N'}{|m'|^2}+\frac{N}{|m|^2}
-
\Bigl(\frac{\sigma}{m}-\frac{\sigma'}{m'}\Bigr)^2
\biggl)
+
\Im\Bigip{\frac{\sigma}{m}}{\frac{\sigma'}{m'}}
+
3\Bigl(\frac{\nu'}{|m'|^2}-\frac{\nu}{|m|^2}\Bigr)
\biggr].
\end{equation}
\emph{Caution:} we are using the convention that the imaginary
part of a complex number is imaginary; for example 
$\Im\theta$ is $\theta$ rather than $\sqrt3$.

In the rest of this section, ``$s$'' will only be used for roots.

\begin{lemma}
\label{lem-classification-of-roots-of-height-theta}
Suppose $\lambda_6,\lambda_9$ are fixed
vectors in $\Leech$ with
norms $6$ and~$9$.  Then
under the  $\G$-stabilizer of $\rho$, every root with $m=\theta$ is
equivalent to  
$(0;\theta,-\w)$ or $(\lambda_6;\theta,\w)$ or
$(\lambda_9;\theta,-1)$.
\end{lemma}

\begin{proof}
The $\G$-stabilizer of $\rho$ contains the Heisenberg group of ``translations''
\begin{align*}
&(l;0,0)\mapsto\bigl(l;0,\thetabar^{-1}\ip{l}{\lambda}\bigr)\\
\llap{$T_{\lambda,z}:$ }&
(0;1,0)\mapsto\bigl(\lambda;1,\theta^{-1}(z-\lambda^2/2)\bigr)\\
&(0;0,1)\mapsto(0;0,1)
\end{align*}
where $\lambda\in\Leech$ and $z\in\Im\C$ are such that
$z-\lambda^2/2\in\theta\E$.  Suppose $s\in L$ has the form
\eqref{eq-definition-of-vector-s} with $N=3$ and $m=\theta$.  Applying
$T_{\lambda,z}$ to $s$ changes the first coordinate by
$\theta\lambda$.  By \cite[p.\ 153]{Wilson}, every element of $\Leech$
is congruent modulo $\theta\Leech$ to a vector of norm $0$, $6$
or~$9$, so we may suppose $\sigma$ has one of these norms.  Since
$\Aut\Leech$ fixes $\rho$ and acts transitively on the vectors of each
of these norms \cite[p.\ 155]{Wilson}, we may suppose $s=0$,
$\lambda_6$ or~$\lambda_9$.  That is, $s$ is one of
$$
\textstyle
(0;\theta,\frac{1}{2}-\nu)
\qquad
(\lambda_6;\theta,-\frac{1}{2}-\nu)
\qquad
(\lambda_9;\theta,-1-\nu).
$$ In each of the three cases, the possibilities for $\nu$ differ by
the elements of $\Im\E$.  Applying $T_{0,z}$ ($z\in\Im\E$) adds $z$ to
the third coordinate of~$s$.  Therefore we may take $\nu=\theta/2$,
$\thetabar/2$ and $0$ in the three cases, yielding the roots in the
statement of the lemma.  (These three roots are
inequivalent under the
$\G$-stabilizer of $\rho$,  but we
don't need this.)
\end{proof}

\begin{lemma}
\label{lem-some-Leech-reflection-moves-cusp-closer-to-projection-point}
Suppose $s$ is the root
$(0;\theta,-\w)$ or a root as in \ref{eq-definition-of-vector-s} with $|m|=2$ or $|m|>\sqrt7$, and
define  
$p$ as the point of $s^\perp$ nearest $\rho$.  Then there is a
triflection in a Leech root that moves $\rho$ closer to $p$.
\end{lemma}

\begin{proof}[Proof of
    lemma~\ref{lem-some-Leech-reflection-moves-cusp-closer-to-projection-point}]
  This proof grew from simpler arguments used for
  \cite[Thm.\ 4.1]{Allcock-Inventiones} and
  \cite[Prop.\ 4.2]{Basak-bimonster-1}.

We have $p=\rho-\frac{1}{3}\ip{\rho}{s}s=\rho+\frac{1}{\theta}\mbar s$.  We
want to choose a Leech root $l$, and $\zeta=\w^{\pm1}$, such that the
$\zeta$-reflection in $l$ (call it $R$) moves $\rho$ closer to
$p$.
This is equivalent to $\ip{p}{R(\rho)}$ being smaller in absolute
value than $\ip{p}{\rho}$.  We will write down these inner products
explicitly  and then choose $l$ and
$\zeta$ appropriately.  Direct calculation gives $\ip{p}{\rho}=-|m|^2$.
Also,
$$
R(\rho)
=
\rho-(1-\zeta)\frac{\ip{\rho}{l}}{\ip{l}{l}}l
=
\rho+\frac{1-\zeta}{\theta}l.
$$

It turns out that the necessary estimates on $\ip{p}{R(\rho)}$ are best
expressed in terms of the following parameter:
\begin{align}
y
:={}&
\frac{\theta}{|m|^2}\ip{p}{l}
=
\frac{\theta}{|m|^2}\Bigip{\rho+\frac{\mbar}{\theta}s}{l}
=
-\frac{3}{|m|^2}+\frac{1}{m}\ip{s}{l}
\label{eq-y-in-terms-of-s-paired-with-l}\\
{}\in{}&
-\frac{3}{|m|^2}+\frac{1}{m}\theta\E.
\label{eq-y-in-certain-lattice}
\end{align}
First one works out 
\begin{equation}
\label{eq-rho-closer-to-p-means-this-quantity-less-than-one}
\biggl|
\frac{\ip{p}{R(\rho)}}{\ip{p}{\rho}}
\biggr|
=
\biggl|
\frac{\ip{p}{R(\rho)}}{|m|^2}
\biggr|
=
\bigl|
{\textstyle\frac{1}{3}}(1-\zetabar)y-1
\bigr|.
\end{equation}
Our goal is to choose $l$ and $\zeta$ so that this is less than $1$.
This is equivalent to $|y-(1-\zeta)|<\sqrt3$.  Because the possibilities for
$\zeta$ are $\w^{\pm1}$, this amounts to being able to choose $l$ so
that $y$ lies in the union $V$ of the open balls in $\C$ of radius
$\sqrt3$ around the points $1-\w$ and~$1-\wbar$.  
So our goal is to choose $l$ such that $y$ lies in the shaded
region in figure~\ref{fig-region-containing-y}.

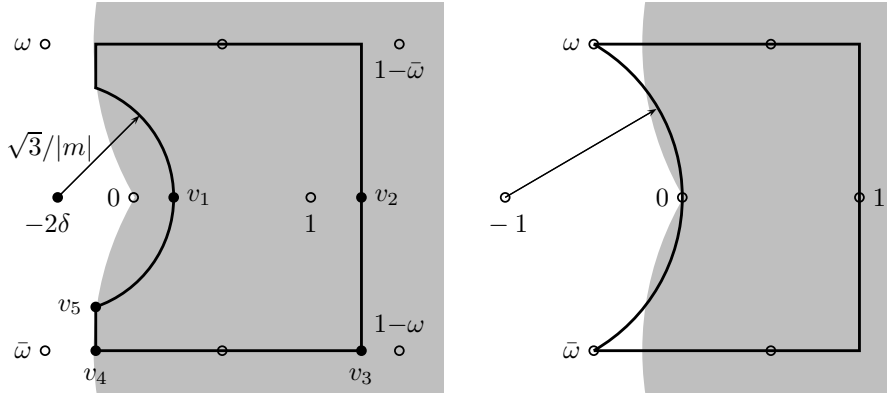
\begin{figure}
\psset{unit=67pt}
\def\URx{1.75}
\def\URy{1.1}
\def\LLx{-.71}
\def\LLy{-1.1}
\begin{pspicture*}(\LLx,\LLy)(\URx,\URy)
\def\pointradius{.03}
\def\pathwidth{.016}
\def\arrowwidth{.008}
\def\framewidth{.01}
\def\pointedgewidth{.01} 
%
%
%
%
\def\Ax{-.5}
\def\Ay{.866}
\def\Bx{.5}
\def\By{.866}
\def\Cx{1.5}
\def\Cy{.866}
\def\Dx{-1}
\def\Dy{0}
\def\Ex{0}
\def\Ey{0}
\def\Fx{1}
\def\Fy{0}
\def\Gx{-.5}
\def\Gy{-.866}
\def\Hx{.5}
\def\Hy{-.866}
\def\Ix{1.5}
\def\Iy{-.866}
\def\centerx{-.429}
\def\centery{0}
\def\arcradius{.655}
\def\maxtheta{71.5}
\def\arrowtipx{.034} 
\def\arrowtipy{.463}
\def\arrowtextx{-.229} 
\def\arrowtexty{.2}
%
%
%
\def\Vx{.226}
\def\Vy{0}
\def\Wx{1.286}
\def\Wy{0}
\def\Xx{\Wx}
\def\Xy{.866}
\def\Yx{-.214}
\def\Yy{.866}
\def\Zx{\Yx}
\def\Zy{.619}
\psset{linestyle=none,fillstyle=solid,fillcolor=lightgray}
\pscircle(\Cx,\Cy){1.732}
\pscircle(\Ix,\Iy){1.732}
\psset{linestyle=solid,linecolor=black,linewidth=\pathwidth, fillstyle=none}
\psline(\Zx,\Zy)(\Yx,\Yy)(\Xx,\Xy)(\Xx,-\Xy)(\Yx,-\Yy)(\Zx,-\Zy)
\psarc(\centerx,\centery){\arcradius}{-\maxtheta}{\maxtheta}
\psset{fillstyle=solid,fillcolor=black}
\pscircle(\Vx,\Vy){\pointradius}
\pscircle(\Wx,\Wy){\pointradius}
\pscircle(\Xx,-\Xy){\pointradius}
\pscircle(\Yx,-\Yy){\pointradius}
\pscircle(\Zx,-\Zy){\pointradius}
\pscircle(\centerx,\centery){\pointradius}
\rput[l](\Vx,\Vy){$\,\,\,v_1$}
\rput[l](\Wx,\Wy){$\,\,\,v_2$}
\rput[t](\Xx,-\Xy){$\vbox to 11pt{}v_3$}
\rput[t](\Yx,-\Yy){$\vbox to 11pt{}v_4$}
\rput[r](\Zx,-\Zy){$v_5\,\,\,$}
\rput[t](\centerx,-\centery){$\vbox to 13pt{}\llap{$-$}2\d$}
\psline[arrows=->,linewidth=\arrowwidth](\centerx,\centery)(\arrowtipx,\arrowtipy)
\rput[br](\arrowtextx,\arrowtexty){$\sqrt3/|m|$}
\psset{linestyle=solid,fillstyle=none,linecolor=black,linewidth=\pointedgewidth}
\pscircle(\Ax,\Ay){\pointradius}
\pscircle(\Bx,\By){\pointradius}
\pscircle(\Cx,\Cy){\pointradius}
\pscircle(\Ex,\Ey){\pointradius}
\pscircle(\Fx,\Fy){\pointradius}
\pscircle(\Gx,\Gy){\pointradius}
\pscircle(\Hx,\Hy){\pointradius}
\pscircle(\Ix,\Iy){\pointradius}
\rput[r](\Ax,\Ay){$\w\,\,\,$}
\rput[t](\Cx,\Cy){$\vbox to 13pt{}1{-}\wbar$}
\rput[r](\Ex,\Ey){$0\,\,\,$}
\rput[t](\Fx,\Fy){$\vbox to 13pt{}1$}
\rput[r](\Gx,\Gy){$\wbar\,\,\,$}
\rput[b](\Ix,\Iy){$\vbox to 13pt{}\raise7pt\hbox{$1{-}\w$}$}
\end{pspicture*}
\def\URx{1.15}
\def\URy{1.1}
\def\LLx{-1.1}
\def\LLy{-1.1}
\begin{pspicture*}(\LLx,\LLy)(\URx,\URy)
\def\pointradius{.03}
\def\pathwidth{.016}
\def\arrowwidth{.008}
\def\framewidth{.01}
\def\pointedgewidth{.01} 
%
%
%
%
\def\Ax{-.5}
\def\Ay{.866}
\def\Bx{.5}
\def\By{.866}
\def\Cx{1.5}
\def\Cy{.866}
\def\Dx{-1}
\def\Dy{0}
\def\Ex{0}
\def\Ey{0}
\def\Fx{1}
\def\Fy{0}
\def\Gx{-.5}
\def\Gy{-.866}
\def\Hx{.5}
\def\Hy{-.866}
\def\Ix{1.5}
\def\Iy{-.866}
\def\centerx{-1}
\def\centery{0}
\def\arcradius{1}
\def\maxtheta{60}
\def\arrowtipx{-.145} 
\def\arrowtipy{.5}
\def\arrowtextx{-.229} 
\def\arrowtexty{.2}
%
%
%
\def\Vx{.226}
\def\Vy{0}
\def\Wx{1}
\def\Wy{0}
\def\Xx{\Wx}
\def\Xy{.866}
\def\Yx{-.5}
\def\Yy{.866}
\def\Zx{\Yx}
\def\Zy{.619}
\psset{linestyle=none,fillstyle=solid,fillcolor=lightgray}
\pscircle(\Cx,\Cy){1.732}
\pscircle(\Ix,\Iy){1.732}
\psset{linestyle=solid,linecolor=black,linewidth=\pathwidth, fillstyle=none}
\psline(\Yx,\Yy)(\Xx,\Xy)(\Xx,-\Xy)(\Yx,-\Yy)
\psarc(\centerx,\centery){\arcradius}{-\maxtheta}{\maxtheta}
\psset{fillstyle=solid,fillcolor=black}
\psline[arrows=->,linewidth=\arrowwidth](\centerx,\centery)(\arrowtipx,\arrowtipy)
\psset{linestyle=solid,fillstyle=none,linecolor=black,linewidth=\pointedgewidth}
\pscircle(\Ax,\Ay){\pointradius}
\pscircle(\Bx,\By){\pointradius}
\pscircle(\Dx,\Dy){\pointradius}
\pscircle(\Ex,\Ey){\pointradius}
\pscircle(\Fx,\Fy){\pointradius}
\pscircle(\Gx,\Gy){\pointradius}
\pscircle(\Hx,\Hy){\pointradius}
\rput[r](\Ax,\Ay){$\w\,\,\,$}
\rput[t](\Dx,\Dy){$\vbox to 13pt{}-1$}
\rput[r](\Ex,\Ey){$0\,\,\,$}
\rput[l](\Fx,\Fy){$\,\,\,1$}
\rput[r](\Gx,\Gy){$\wbar\,\,\,$}
\end{pspicture*}
\caption{See the proof of lemma~\ref{lem-some-Leech-reflection-moves-cusp-closer-to-projection-point}.
  $V$ is the union of the gray (open) disks, which have radius
  $\sqrt3$ and centers $1-\w^{\pm1}$.  We seek a Leech root $l$ so that
  $y$ lies in this region.  $U$ is the closed region bounded by the
  solid line, and is where we can arrange for $y$ to be.  $U$ varies with $|m|$; we have drawn the case $|m|=\sqrt7$, when $v_5$ is on
  the boundary of $V$, and the case $|m|=\sqrt3$, when $v_4$ and $v_5$
  coalesce at~$\wbar$.  Hollow circles indicate Eisenstein integers.}
\label{fig-region-containing-y}
\end{figure}

Now we examine how our choice of $l$ affects $y$.  Writing $l$ as in
\eqref{eq-definition-of-a-leech-root-l}, choosing it
amounts to
choosing
$\lambda\in\Leech$, and then choosing $\nu_l\in\Im\C$ subject to the condition
that the last coordinate of \eqref{eq-definition-of-a-leech-root-l} is in $\E$.
Specializing \eqref{eq-general-inner-product-formula} to the case that
$s$ has norm $N=3$ and $s'$ is the Leech root $l$ gives
$$
\ip{s}{l}
=
m
\biggl[
\frac{3}{2|m|^2}
+
\frac{3}{2}
-
\frac{1}{2}\Bigl(\frac{\sigma}{m}-\lambda\Bigr)^2
+
\Im\Bigip{\frac{\sigma}{m}}{\lambda}
+
3\Bigl(\nu_l-\frac{\nu}{|m|^2}\Bigr)
\biggr].
$$  
Plugging this into formula \eqref{eq-y-in-terms-of-s-paired-with-l} gives
\begin{equation}
y
{}=
\label{eq-formula-for-y}
-\frac{3}{2|m|^2}+\frac{3}{2}-\frac{1}{2}\Bigl(\frac{\sigma}{m}-\lambda\Bigr)^2
+
\Im\Bigip{\frac{\sigma}{m}}{\lambda}+3\Bigl(\nu_l-\frac{\nu}{|m|^2}\Bigr).
\end{equation}

The covering radius of a lattice in Euclidean space is defined as the
smallest number such that the closed balls of that radius, centered at
lattice points, cover Euclidean space.  The covering radius of $\Leech$
is $\sqrt3$, because the underlying real lattice has norms equal to
$3/2$ times those of the real Leech lattice, whose covering radius is
$\sqrt2$ by \cite{Conway-Parker-Sloane}.  Therefore we may take
$\lambda$ so that $0\leq(\sigma/m-\lambda)^2\leq3$.  It follows that
the real part of \eqref{eq-formula-for-y} lies in $[-\d,3/2-\d]$ where
$\d:=3/2|m|^2$.

Next we choose $\nu_l$.  The only constraint on it is that the
last component of $l=(\lambda;1,?)$ must lie in $\E$.  As mentioned
after \eqref{eq-definition-of-a-leech-root-l}, this amounts
to $\nu_l\in\frac{1}{\theta}(\frac12+\Z)$ if $\lambda^2$ is divisible by~$6$,
and $\nu_l\in\frac{1}{\theta}\Z$ otherwise.  In either case,
referring to \eqref{eq-formula-for-y} shows that changing our choice of $\nu_l$ allows
us to change $y$ by any rational integer multiple
of $\theta$.
So
we may take 
$\Im y\in[-\theta/2,\theta/2]$.
After
these choices we have
\begin{equation}
\label{eq-rectangle-containing-y}
\Re y\in[-\d,3/2-\d\bigr]
\quad\hbox{and}\quad
\Im y\in 
[-\theta/2,\theta/2].
\end{equation}

Now we can derive additional information about $y$.  We have $y\neq-2\d$ since $-2\d$ is not in the
rectangle \eqref{eq-rectangle-containing-y}, and
since $y\in-2\d+\frac{\theta}{m}\E$ by
\eqref{eq-y-in-certain-lattice}, $y$ lies at
distance${}\geq\sqrt3/|m|$ from~$-2\d$.  We define $U$ as the closed
rectangle \eqref{eq-rectangle-containing-y} in $\C$ minus the open
$(\sqrt3/|m|)$-disk around $-2\d$.  
We have shown that we may choose a Leech root $l$ such that $y\in U$.
We have indicated $U$ in outline in figure~\ref{fig-region-containing-y}; as $|m|$ increases, the rectangle
moves to the right, 
the center of the removed disk approaches zero, and its radius
approaches zero more slowly than the center does.

Now suppose $|m|^2>7$.  We claim $U\sset V$.  Assuming this for the
moment, we may choose $l$ such that $y$ is in $U$, hence $V$, which
allows us to choose $\zeta=\w^{\pm1}$ so that $\zeta$-reflection in
$l$ moves $\rho$ closer to $p$.  This finishes the proof.
To prove the claim it will suffice to show that the lower half of $U$
lies in the open $\sqrt3$-ball around $1-\w$.  Obviously it suffices
to check this for the points marked $v_1,\dots,v_5$ in
figure~\ref{fig-region-containing-y}.  These are
$v_1=-2\d+\sqrt3/|m|$, $v_2=\frac{3}{2}-\d$,
$v_3=\frac{3}{2}-\d-\frac{1}{2}i\sqrt3$, $v_4=-\d-\frac{1}{2}i\sqrt3$
and
$$
v_5=-\d
-i
\sqrt{\frac{3}{|m|^2}-\frac{9}{4|m|^4}}.
$$ 
Using $|m|^2>7$, one can check that each of these lies at
distance${}<\sqrt3$ from $1-\w$.
This finishes the proof of the $|m|^2>7$ case. (If $|m|^2=7$ then
$v_5$ lies in the boundary of $V$. 
If $|m|^2=4$ then $v_4$ and $v_5$ are outside the boundary; see figure~\ref{fig-m=2-case}.  If
$|m|^2=3$ then $v_1=0$ is on the boundary and $v_4=v_5=\wbar$ is
outside it; see the second part of figure~\ref{fig-region-containing-y}.)

Next we treat the special case $s=(0;\theta,-\w)$.  Choosing $\lambda=0$
gives $\Re y=1$ by \eqref{eq-formula-for-y}.  
Then choosing $\nu_l$ as above, so that $\Im y$ lies in
$[-\theta/2,\theta/2]$, yields
$y\in V$.  So we can move $\rho$ closer to $p$ just as in the
$|m|^2>7$ case.

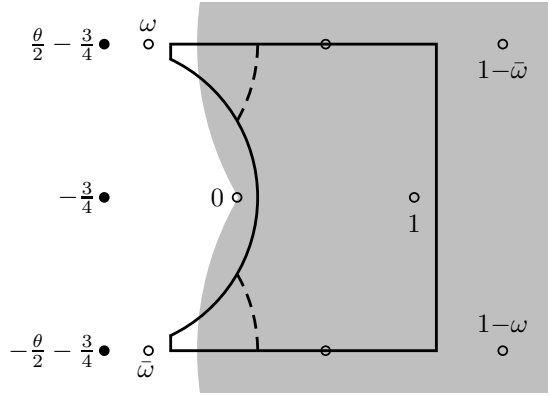
\begin{figure}
\psset{unit=67pt}
\def\URx{1.75}
\def\URy{1.1}
\def\LLx{-1.4}
\def\LLy{-1.1}
\begin{pspicture*}(\LLx,\LLy)(\URx,\URy)
\def\pointradius{.03}
\def\pathwidth{.016}
\def\arrowwidth{.008}
\def\framewidth{.01}
\def\pointedgewidth{.01} 
%
%
%
%
\def\Ax{-.5}
\def\Ay{.866}
\def\Bx{.5}
\def\By{.866}
\def\Cx{1.5}
\def\Cy{.866}
\def\Dx{-1}
\def\Dy{0}
\def\Ex{0}
\def\Ey{0}
\def\Fx{1}
\def\Fy{0}
\def\Gx{-.5}
\def\Gy{-.866}
\def\Hx{.5}
\def\Hy{-.866}
\def\Ix{1.5}
\def\Iy{-.866}
\def\centerx{-.75}
\def\centery{0}
\def\arcradius{.866}
\def\maxtheta{64.8}
\def\maxphi{30} 
\def\arrowtipx{.034} 
\def\arrowtipy{.463}
\def\arrowtextx{-.229} 
\def\arrowtexty{.2}
%
%
%
\def\Vx{.116}
\def\Vy{0}
\def\Wx{1.125}
\def\Wy{0}
\def\Xx{\Wx}
\def\Xy{.866}
\def\Yx{-.375}
\def\Yy{.866}
\def\Zx{\Yx}
\def\Zy{.781}
\psset{linestyle=none,fillstyle=solid,fillcolor=lightgray}
\pscircle(\Cx,\Cy){1.732}
\pscircle(\Ix,\Iy){1.732}
\psset{linestyle=solid,linecolor=black,linewidth=\pathwidth, fillstyle=none}
\psline(\Zx,\Zy)(\Yx,\Yy)(\Xx,\Xy)(\Xx,-\Xy)(\Yx,-\Yy)(\Zx,-\Zy)
\psarc(\centerx,\centery){\arcradius}{-\maxtheta}{\maxtheta}
\psarc[linestyle=dashed](\centerx,-.866){\arcradius}{0}{\maxphi}
\psarc[linestyle=dashed](\centerx,.866){\arcradius}{-\maxphi}{0}
\psset{fillstyle=solid,fillcolor=black}
\pscircle(\centerx,\centery){\pointradius}
\pscircle(\centerx,.866){\pointradius}
\pscircle(\centerx,-.866){\pointradius}
\rput[r](\centerx,-\centery){$-\frac{3}{4}\,\,$}
\rput[r](\centerx,.866){$\frac{\theta}{2}-\frac{3}{4}\,\,$}
\rput[r](\centerx,-.866){$-\frac{\theta}{2}-\frac{3}{4}\,\,$}
\psset{linestyle=solid,fillstyle=none,linecolor=black,linewidth=\pointedgewidth}
\pscircle(\Ax,\Ay){\pointradius}
\pscircle(\Bx,\By){\pointradius}
\pscircle(\Cx,\Cy){\pointradius}
\pscircle(\Ex,\Ey){\pointradius}
\pscircle(\Fx,\Fy){\pointradius}
\pscircle(\Gx,\Gy){\pointradius}
\pscircle(\Hx,\Hy){\pointradius}
\pscircle(\Ix,\Iy){\pointradius}
\rput[b](\Ax,\Ay){$\vbox to 13pt{}\raise5pt\hbox{$\w$}$}
\rput[t](\Cx,\Cy){$\vbox to 13pt{}1{-}\wbar$}
\rput[r](\Ex,\Ey){$0\,\,\,$}
\rput[t](\Fx,\Fy){$\vbox to 13pt{}1$}
\rput[t](\Gx,\Gy){$\vbox to 10pt{}\wbar\,$}
\rput[b](\Ix,\Iy){$\vbox to 13pt{}\raise7pt\hbox{$1{-}\w$}$}
\end{pspicture*}
\caption{The analogue of figure~\ref{fig-region-containing-y} for the special case $|m|=2$ in
  the proof of lemma~\ref{lem-some-Leech-reflection-moves-cusp-closer-to-projection-point}.
  The proof shows that $y$ lies in $U$ (bounded by the solid path) but outside two open disks
  (indicated by the dashed arcs), hence in $V$ (the shaded region).}
\label{fig-m=2-case}
\end{figure}

Finally, we suppose $|m|=2$; we may take $m=2$ by multiplying $s$ by a
unit.  Recall that once we proved that $y$ lies in the rectangle
\eqref{eq-rectangle-containing-y}, we could use \eqref{eq-y-in-certain-lattice} to show that $y$ lies outside the open
disk used in the definition of~$U$.  For $m=2$ the argument shows more.
Since $\delta=\frac{3}{8}$ when $|m|=2$, \eqref{eq-y-in-certain-lattice} shows
$y\in-\frac{3}{4}+\frac{\theta}{2}\E$.  Since
$-\frac{3}{4}\pm\frac{\theta}{2}$ lie in
$-\frac{3}{4}+\frac{\theta}{2}\E$ but not in the rectangle \eqref{eq-rectangle-containing-y},
$y$ lies at distance${}\geq\sqrt3/2$ from each of them, just as it
lies at distance${}\geq\sqrt3/2$ from $-\frac{3}{4}$.  It is easy to
check that $U$ minus the open $\sqrt3/2$-balls around
$-\frac{3}{4}\pm\frac{\theta}{2}$ lies in~$V$; see figure~\ref{fig-m=2-case}.
Therefore $y\in V$, finishing the proof as before.  (One can consider
analogues of these extra disks for any $m$.  
They are unnecessary if $|m|^2>7$, and  turn out to be
useless if $|m|^2=3$ or~$7$.)
\end{proof}

\begin{lemma}
\label{lem-height-reduction-almost-moves-rho-closer-to-projection-point}
Suppose $s$ is the root $(\lambda_6;\theta,\w)$ or a root as in \ref{eq-definition-of-vector-s}
with
$|m|=\sqrt7$, and define $p$ as the point of $s^\perp$ nearest $\rho$.
Then there is a triflection in a Leech root that either moves $\rho$
closer to $p$, or else moves $\rho$ closer to $s^\perp$ while
preserving $\rho$'s distance from~$p$.
\end{lemma}

\begin{proof}
Suppose first $|m|=\sqrt7$.  Then the proof of lemma~\ref{lem-some-Leech-reflection-moves-cusp-closer-to-projection-point} goes
through unless $y$ is $v_5$ in figure~\ref{fig-region-containing-y}, or its complex
conjugate. So suppose $y=v_5$ or $\bar{v}_5$, and take $\zeta=\w$ or
$\wbar$ respectively.  The argument in the proof of lemma~\ref{lem-some-Leech-reflection-moves-cusp-closer-to-projection-point}, that  $R$ moves $\rho$
closer to $p$, fails because $\bigl|y-(1-\zeta)\bigr|$ equals $\sqrt3$
rather than being strictly smaller.
But it does show that $R(\rho)$ is
exactly as far from $p$ as $\rho$ is.  This is one of our claims, and
what remains to show is that $R$ moves $\rho$  closer to $s^\perp$.

To do this we first solve \eqref{eq-y-in-terms-of-s-paired-with-l} for
$\ip{s}{l}$ in terms of $y$, obtaining $\ip{s}{l}=(3+|m|^2y)/\mbar$.
Then one works out
$$
\biggl|
\frac{\ip{s}{R(\rho)}}{\ip{s}{\rho}}
\biggr|
=
\biggl|
\frac{\ip{s}{\rho}-{\textstyle\frac{1}{\theta}}(1-\zetabar)\ip{s}{l}}{\ip{s}{\rho}}
\biggr|
=
\biggl|
1-
\frac{1}{3}
(1-\zetabar)
\Bigl(
\frac{3}{|m|^2}+y
\Bigr)
\biggr|.
$$
We want this to be less than~$1$.
By copying the argument following \eqref{eq-rho-closer-to-p-means-this-quantity-less-than-one}, this is equivalent to
$y+3/|m|^2$ lying in the open \hbox{$\sqrt3$-disk} around $1-\zeta$.  This is
obvious from the figure because $y+3/|m|^2$ is $3/7$ to the right of
$y=v_5$ or $\bar{v}_5$. This finishes the $|m|=\sqrt7$ case.  

The case $s=(\lambda_6;\theta,\w)$ is similar.  In this case $U$
appears in figure~\ref{fig-region-containing-y}.  
Writing l as in \eqref{eq-definition-of-a-leech-root-l} with $\lambda=0$ leads to $\Re y=0$, so
either $y\in V$ (so the proof of
lemma~\ref{lem-some-Leech-reflection-moves-cusp-closer-to-projection-point}
applies) or else $y=0\in\partial V$.
In this case the argument for  $|m|=\sqrt7$ shows that
$R(\rho)$ is just as close to $p$ as $\rho$ is, and that $R$ moves $\rho$
closer to $s^\perp$.
\end{proof}

\begin{lemma}
\label{lem-hole-meridians-generated-by-leech-meridians}
Let $s=(\lambda_9;\theta,-1)$, define $p$ as the point of $s^\perp$
nearest~$\rho$, and $B$ as the open horoball centered at $\rho$,
whose bounding horosphere is tangent to $s^\perp$ at $p$. 
Then there exists an open ball $U$ around $p$ with
$U\cap\H=U\cap\H_p$, and a triflection $R$ in one of the Leech
mirrors, such that $B\cap R(B)\cap U\neq\emptyset$ and $R(B)\cap U\cap
s^\perp\neq\emptyset$. 
\end{lemma}

\begin{proof}
Since we are verifying hypothesis \eqref{item-more-complicated-move-closer-condition} of lemma~\ref{lem-lowering-lemma}, we will
use that lemma's notation $H$ for $s^\perp$.  By definition,
$$
\textstyle
p
=\rho-\frac{1}{3}\ip{\rho}{s}s
=(-\lambda_9;\thetabar,2).
$$
This has norm~$-3$ and lies in $L$.  One computes
$\height_\rho(p)=3$, so $B$ is the height~$3$ open horoball around $\rho$.
We take 
$U$ to have radius 
$\sinh^{-1}\sqrt{1/3}$.  
To check that $U\cap\H=U\cap\H_p$, consider a root $s'$ not
orthogonal to $p$.  Then 
$\bigl|\ip{p}{s'}\bigr|\geq\sqrt3$ since $p\in L$, so 
$$
d(p,{s'}^\perp)
=
\sinh^{-1}\sqrt{-\frac{\bigl|\ip{p}{s'}\bigr|^2}{p^2{s'}^2}}
\geq
\sinh^{-1}\sqrt{1/3},
$$
as desired.

Next we choose $R$ to be the $\w$-reflection in the Leech root
$l=(0;1,-\w)$.  (We found $l$ by applying the proof of lemma~\ref{lem-some-Leech-reflection-moves-cusp-closer-to-projection-point} as
well as we could.  That is, we choose $l$ so that $y$ in that proof
equals the lower left corner $\wbar$ of the second part of
figure~\ref{fig-region-containing-y}.)  This yields
$R(\rho)=(0;\wbar,0)$.  We must verify $R(B)\cap U\cap H\neq\emptyset$
and $B\cap R(B)\cap U\neq\emptyset$.

Our strategy for $R(B)\cap U\cap H\neq\emptyset$ is to begin by
defining $p'$ as the projection of $R(\rho)$ to $H$, which turns out
to lie outside~$U$.  Then we parameterize $\geodesic{p' p}\sset H$,
find the point $x$ where it crosses $\partial U$, and check that $x\in
R(B)$.  So $x\in R(B)\cap H\cap\partial U$.  Therefore a point of $\geodesic{p'
  p}$, slightly closer to $p$ than $x$ is, lies in
$R(B)\cap H\cap U$, showing that this intersection is nonempty.  

Here are the details.  Computation gives
$p'=(\wbar\lambda_9/\theta;2\wbar,-\wbar/\theta)$, of norm~$-1$.
One checks $\ip{p'}{p}=2\wbar\thetabar$, so
$d(p,p')=\cosh^{-1}2>\sinh^{-1}\sqrt{1/3}$ and $p'$ lies outside $U$,
as claimed.  Also, $-\w\theta p'$ and $p$ have negative inner product.
Therefore $\geodesic{p' p}-\{p\}$ is parameterized by $x_t=-\w\theta
p'+t p$ with $t\in[0,\infty)$.  One computes $\ip{x_t}{p}=-3t-6$ and
  $x_t^2=-3t^2-12t-3$, yielding
$$
d(x_t,p)=\cosh^{-1}\sqrt{\frac{\bigl|\ip{x_t}{p}\bigr|^2}{x_t^2 p^2}}
=\cosh^{-1}\sqrt{\frac{t^2+4t+4}{t^2+4t+1}}
$$
Now, $x_t$ lies in $\partial U$ just when this equals
$\sinh^{-1}\sqrt{1/3}$, yielding a quadratic equation for~$t$.  There
is just one nonnegative solution, namely $t=2\sqrt3-2$.  So
$x=x_{2\sqrt3-2}$.  Then one computes
$\ip{R(\rho)}{x}=\wbar\thetabar(4\sqrt3-3)$, so
$$
\height_{R(\rho)}(x)=-\frac{\bigl|\ip{R(\rho)}{x}\bigr|^2}{x^2}
=-\frac{3\bigl(57-24\sqrt3\,\bigr)}{-27}<3.
$$
That is, $x\in R(B)$ as desired.

Our strategy for $B\cap R(B)\cap U\neq\emptyset$ is similar.  We
parameterize the geodesic $\geodesic{x\rho}$, where $x$ is the point
found in the previous paragraph, find the point $y$ where it crosses
$\partial B$, and check that $y$ lies in $R(B)$ and $U$.  Here are the
details.  Computation shows $\ip{x}{\rho}=-6\sqrt3<0$, so
$\geodesic{x\rho}-\{\rho\}$ is parameterized by $y_u=x+u\rho$ with
$u\in[0,\infty)$.  Further computation shows $\ip{y_u}{\rho}=-6\sqrt3$
  and $y_u^2=-27-12u\sqrt3$, so $\height_\rho(y_u)=36/(9+4u\sqrt3)$.
  Setting this equal to~$3$ yields $u=\sqrt3/4$, so $y=y_{\sqrt3/4}$.
  Now one checks that $\height_{R(\rho)}(y)<3$, so that $y\in R(B)$.
  A similar calculation proves $y\in U$.  (In fact this calculation
  can be omitted, because $y,p$ are the projections to $\partial B$ of
  the two points $x,p$ outside $B$, but not both in $\partial B$. 
  Projection to a closed horoball decreases the distance between two
  points, if at least one of them is outside it.  Therefore
  $d(y,p)<d(x,p)=\sinh^{-1}\sqrt{1/3}$.)
\end{proof}

\begin{proof}[Proof of
    theorem~\ref{thm-leech-generation}] We will mimic the
  proof of theorem~\ref{thm-not-quite-strong-enough} (see the end of section~\ref{sec-loops-in-quotients}), using
  lemmas~\ref{lem-some-Leech-reflection-moves-cusp-closer-to-projection-point}--\ref{lem-hole-meridians-generated-by-leech-meridians} in place of the ``moves $a$ closer to $p$''
  hypothesis of that theorem.  Write $G$ for the subgroup of
  $G_a=\piorb\bigl((\ch^{13}-\H)/\PG,a\bigr)$ generated by the Leech
  meridians, i.e., the pairs $(\mu_{a,A,l^\perp},R_l)$ with $l$ a
  Leech root.  We must show that $G$ is all of $G_a$.  It is known
  (\cite{Basak-bimonster-1}, or \cite{Allcock-y555} for a later proof) that the $R_l$'s generate $\PG$.  By the exact
  sequence \eqref{eq-exact-sequence-on-pi-1}, it therefore suffices to show that $G$ contains
  $\pi_1(\ch^{13}-\H,a)$.  By theorem~\ref{thm-generators-for-metric-neighborhood-of-A} it suffices to show that
  $G$ contains every $\Loop{AH}$, with $H$ varying over $\M$.  We do
  this by induction on the distance from $H$ to $\rho$, or properly
  speaking, on $\bigl|\ip{\rho}{s}\bigr|$ where $s$ is a root with
  $H=s^\perp$.  The base case is when $s$ is a Leech root, i.e.,
  $\bigl|\ip{\rho}{s}\bigr|=\sqrt3$, and we just
  observe $\Loop{AH}=(\mu_{a,A,H},R_s)^3$.

Now suppose $s$ is a root but not a Leech root, $H=s^\perp$, $p$ is
the point of $H$ closest to $\rho$, and $B$ is the open horoball
centered at $\rho$ and tangent to $H$ at $p$.  We may assume by
induction that $G$ contains every $\Loop{A,s'^\perp}$ with $s'$ a root
satisfying $\bigl|\ip{\rho}{s'}\bigr|<\bigl|\ip{\rho}{s}\bigr|$.  It
follows from theorem~\ref{thm-generators-for-metric-neighborhood-of-A} that $G$ contains $\pi_1(B-\H,a)$.

The smallest possible value of $\bigl|\ip{\rho}{s}\bigr|$ for a
non-Leech root $s$ is~$3$, occurring when $|m|=\sqrt3$ in \eqref{eq-definition-of-vector-s}. In
the  cases $s=(0;\theta,-\w)$, $(\lambda_6;\theta,\w)$,
resp.\ $(\lambda_9;\theta,-1)$, hypothesis \eqref{item-can-move-A-closer-to-p}, \eqref{item-can-move-A-closer-to-H-and-no-further-from-A},
resp.\ \eqref{item-more-complicated-move-closer-condition} of lemma~\ref{lem-lowering-lemma} is satisfied, by lemma~\ref{lem-some-Leech-reflection-moves-cusp-closer-to-projection-point}, \ref{lem-height-reduction-almost-moves-rho-closer-to-projection-point},
resp.~\ref{lem-hole-meridians-generated-by-leech-meridians}.  If~$s$ is any root with $\ip{\rho}{s}=3$ then it is
equivalent to one of these examples under the $\G$-stabilizer
of~$\rho$, by lemma~\ref{lem-classification-of-roots-of-height-theta}.  
Therefore lemma~\ref{lem-lowering-lemma} applies to
$s^\perp$ for every root $s$ with 
$\ip{\rho}{s}=3$.  It
follows that $G$ contains the corresponding loops $\Loop{AH}$.
If $\bigl|\ip{\rho}{s}\bigr|=3$ then scaling $s$ by a unit reduces to the 
$\ip{\rho}{s}=3$ case.

The next possible value of $\bigl|\ip{\rho}{s}\bigr|$ is $2\sqrt3$,
occurring when $|m|=2$ in \eqref{eq-definition-of-vector-s}.  In this case lemma~\ref{lem-some-Leech-reflection-moves-cusp-closer-to-projection-point} verifies
hypothesis \eqref{item-can-move-A-closer-to-p} of lemma~\ref{lem-lowering-lemma}, which tells us that $G$ contains
$\Loop{AH}$.  The next possible value of $\bigl|\ip{\rho}{s}\bigr|$ is
$\sqrt{21}$, occurring when $|m|=\sqrt7$.  In this case lemma~\ref{lem-height-reduction-almost-moves-rho-closer-to-projection-point}
verifies hypothesis \eqref{item-can-move-A-closer-to-H-and-no-further-from-A} of lemma~\ref{lem-lowering-lemma}, which tells us that $G$
contains $\Loop{AH}$.  
The general step of the induction is essentially the same. 
If $\bigl|\ip{\rho}{s}\bigr|$ is larger than
$\sqrt{21}$, then $|m|$ is larger than $\sqrt7$, so
lemma~\ref{lem-some-Leech-reflection-moves-cusp-closer-to-projection-point} verifies hypothesis \eqref{item-can-move-A-closer-to-p} of lemma~\ref{lem-lowering-lemma}. This tells
us that $G$ contains $\Loop{AH}$, completing the inductive step.
\end{proof}

\end{document}